\documentclass[preprint]{article}
\usepackage{amsfonts}
\usepackage{amssymb}
\usepackage{amsmath}
\usepackage{lineno}
\usepackage{graphicx}
\usepackage{color}
\usepackage{anysize}
\usepackage{algorithmic}
\usepackage{float}
\usepackage{footnote}
\usepackage{subfigure}
\usepackage{multirow}
\usepackage{slashbox}
\usepackage{longtable}

\newtheorem{theorem}{Theorem}

\newtheorem{corollary}[theorem]{Corollary}

\newtheorem{lemma}[theorem]{Lemma}

\newtheorem{remark}[theorem]{Remark}

\newenvironment{proof}[1][Proof]{\noindent\textbf{#1.} }{\ \rule{0.5em}{0.5em}}

\begin{document}

\title{Numerical analysis of a finite element formulation of the P2D model for Lithium-ion cells}
\author{R. Bermejo \\
{\small
$$
Dpto. Matem\'{a}tica Aplicada a la Ingenier\'{i}a Industrial
ETSII Universidad Polit\'{e}cnica de Madrid. }\\
[0pt] {\small e-mail: rodolfo.bermejo@upm.es }}
\date{}
\maketitle

\begin{abstract}
The mathematical P2D model is a system of strongly coupled nonlinear
parabolic-elliptic equations that describes the electrodynamics of
lithium-ion batteries. In this paper, we present the numerical
analysis of a finite element-implicit Euler scheme for such a model.
We obtain error estimates for both the spatially semidiscrete and
the fully discrete systems of equations, and establish the existence
and uniqueness of the fully discrete solution.
\end{abstract}

\textit{2010:MSC:} 65M60, 35M13, 35Q99.

%

\section{Introduction}

In this paper, we present the numerical analysis of a finite
element-implicit Euler method to calculate the numerical solution of
the so called pseudo-two-dimensional (P2D) model. proposed by J.
Newman and coworkers \cite{Doyle}. This is a mathematical model
based on the electrochemical kinetics and continuun mechanics laws,
which consists of a system of coupled nonlinear parabolic-elliptic
equations to model the physical-chemical phenomena governing the
behavior of lithium ion batteries. The P2D model is very much used
in engineering studies. A good presentation of it can be found in
\cite{NT} and \cite{Plet}. A lithium-ion battery system is composed
of a number of lithium-ion cells. A typical cell consists of three
regions, namely, a porous negative electrode (which plays the role
of anode of the cell in the discharge process) connected to the
negative terminal collector of the battery, a separator that is an
electron insulator allowing the flow of lithium ions between the
anode and the cathode, and a porous positive electrode (which plays
the role of cathode during the discharge process) connected to the
positive terminal, see Fig. 1. We must point out that in the charge
process the negative electrode plays the role of cathode and the
positive electrode is the anode. The electrodes are composite porous
structures of highly packed active lithium particles, typically
Li$_{\mathrm{x}}$C$_{\mathrm{6}}$ in the negative electrode and
metal oxide, such as
Li$_{\mathrm{1-x}}$Mn$_{\mathrm{2}}O_{\mathrm{4}}$, in the positive
electrode, plus a binder and a polymer that act as conductive
agents. Furthermore, the cell is filled with the electrolyte that
occupies the holes left free by the particles and the filler
material. The electrolyte is a lithium salt dissolved in an organic
solvent. In the description of the model it is customary to consider
two phases: the electrolyte phase and the solid phase, the latter is
composed of the solid particles of the electrodes.

\begin{figure}[th]
\begin{center}
\includegraphics[height=10cm]{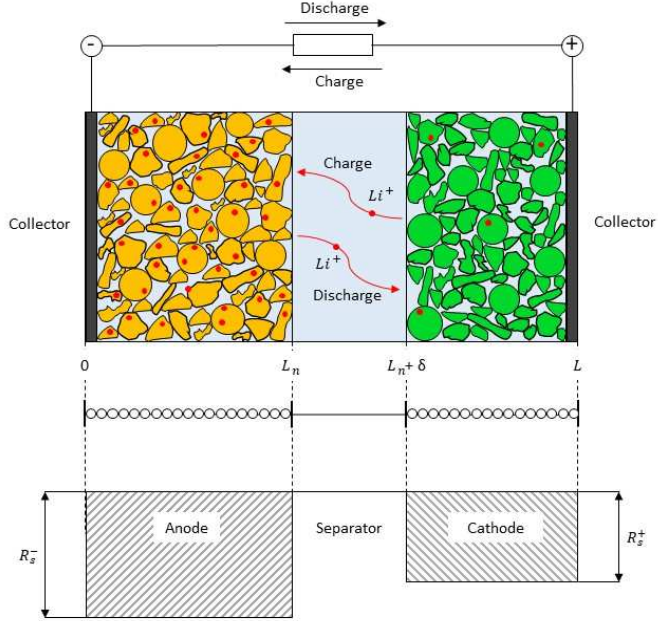}
\end{center}
\caption{\textit{Upper panel}: a cross-section of a cell along the
x-direction. The lithium ions travel from the anode to the cathode
during the discharge process and in the opposite direction during
the charge process. \textit{Middel panel}: the cell model as a
non-denumerable collection of solid spheres plus the separator,
there is one sphere of radius $R_{s}(x)$
at each point ${x}$ of the electrodes. \textit{Lower panel}: the domain $%
D_{3}$ }
\label{figure1}
\end{figure}

The P2D model of a lithium-ion cell considers that the dynamics is only
relevant along the $x$-axis, neglecting what happens along the $y$-axis and $%
z$-axis, because the ratios $\frac{L_{x}}{L_{y}}$ and $\frac{L_{x}}{L_{z}}%
=O(10^{-3})$, $L_{x}$,\ $L_{y}$ and $L_{z}$ being the characteristic
length scales along the corresponding axes. The main modeling
assumptions are the following : (1) The active particles of the
electrodes are assumed to be spheres of radius $R_{s}$ which may be
different in each electrode. (2)\ Side reactions are neglected and
no gas phase is present. (3) The transport of lithium ions is due to
diffusion and migration in the electrolyte solution, and in the
solid particles the atoms of lithium move between vacancies in the
crystalline structure of the particles due to local diffusion in
concentration. By longitudinal and latitudinal symmetry
considerations, the diffusion in the active particles is only in the
radial direction. (4) The electrochemical reaction of lithium
insertion and extraction processes follows the Buttler-Volmer law.
(5)\ The effective transport coefficients are calculated by the
Bruggeman relation, i.e., $\mu ^{eff}=\mu \varepsilon ^{p}$ (p=1.5),
where $\mu $ is a generic transport coefficient and $\varepsilon $
is the component volume fraction of the material in the composite
electrodes and separator.To formulate the
equations of the model we distinguish the following domains.%
\begin{equation*}
\left\{
\begin{array}{l}
D_{\mathrm{n}}=(0,L_{n}),\ D_{\mathrm{s}}=(L_{n},L_{n}+\delta
)\text{,\ }D_{\mathrm{p}}=(L_{n}+\delta ,L),\ L_{p}:=L-(L_{n}+\delta) , \\
\\
D_{1}=(0,L),\ D_{2}=D_{\mathrm{n}}\cup D_{\mathrm{p}}\text{ \
\textrm{and\
\ }}D_{3}=\cup _{x\in D_{2}}\left\{ x\right\} \times (0,R_{\mathrm{s}}(x)),%
\end{array}%
\right.
\end{equation*}%
where $D_{\mathrm{n}}$, $D_{\mathrm{s}}$ and $D_{\mathrm{p}}$ denote
the domains of the negative electrode, the separator and the
positive electrode respectively. Notice that $D_{1}$ represents the
cell domain, $D_{2}$ is a domain that is the union of two disjoint
domains corresponding to the electrodes, and $D_{3}$ is in a certain
sense a modeling space accounting for the spherical balls of radius
$R_{\mathrm{s}}(x)$ that represent at each
$x\in D_{2}$ the solid active particles, such that when $x\in D_{\mathrm{n}}$%
, $R_{\mathrm{s}}(x)=R_{\mathrm{s}}^{-}$, and when $x\in $
$D_{\mathrm{p}}$, $R_{\mathrm{s}}(x)=R_{\mathrm{s}}^{+}$. The
variables of the model are the following:\ for the electrolyte
phase, the molar concentration of lithium ions $u(x,t)$, and the
electric potential $\phi _{1}(x,t)$, $x\in D_{1}$; for the solid
phase, the molar concentration of lithium $v(x;r,t)$, $x\in D_{2}$
and $r\in \left( 0;R_{\mathrm{s}}(x)\right) $, and the electric
potential $\phi _{2}(x,t)$, $x\in D_{2}$. Another important variable
is the
so called molar flux of lithium ions exiting the solid particles, $%
J(x,u,v,\phi _{1},\phi _{2},U)/F$, $F$ being the Faraday constant.
The mathematical expression of $J$ is given by the Buttler-Volmer
law, see (\ref{Ja}).

Many numerical models to integrate the P2D model have been proposed.
The first one is the Dualfoil model developed by J. Newman and his
collaborators \cite{dual}, this is a model that uses second order
finite differences for space discretization of the differential
operators combined with the first order backward Euler time stepping
scheme; the Dualfoil model is distributed as free software, which is
being updated through time. Later on, authors such as \cite{ML} and
\cite{SW}, just to cite a few, have developed their own codes by
using second order finite volume for space discretizations combined
with the first order in time implicit Euler scheme for time
discretization. Other authors make the numerical simulations with
COMSOL multi-physics package that uses finite elements for space
discretizations of the equations, the resulting system of nonlinear
differential equations is integrated by different time stepping
schemes, in particular, conventional DAE solvers, such as DASK
\cite{NOR}. New numerical models have recently been proposed to
improve the computational efficiency, to this respect, we mention
the operator splitting technique of \cite{Farkas}, the orthogonal
collocation method for space discretization combined with the first
order implicit Euler scheme for time discretization of \cite{KZ},
and the implicit-explicit Runge-Kutta-Chebyshev finite element
method of \cite{BG}. Despite the activity in the development of
numerical methods no rigorous numerical analysis of such methods has
been published so far; so, to the best of our knowledge, this is the
first paper presenting the analysis of a numerical method developed
to integrate the P2D model.

The layout of the paper is the following. In Section 2 we introduce
the governing equations of the P2D model together with the
functional framework needed for the numerical analysis. Section 3 is
devoted to the semidiscrete space discretization of the model in a
finite element framework. The error analysis of the semi-discrete
solution is performed in Section 4. Since this analysis is long,
then we have split the section into three subsections in order to
make more palatable its presentation. Subsection 4.1 is a collection
of auxiliary results; subsections 4.2 and 4.3 deal with the error
estimates for the potentials and the concentrations, respectively.
The fully discrete model and its error analysis is presented in
Section 5, which is also split into subsections. Since the fully
discrete model is a nonlinear system of elliptic and fully discrete
parabolic equations at each time instant $t_{n}$, then we have also
studied the existence and uniqueness of the solution by applying
Minty-Browder theorem \cite{Zeid} for the elliptic equations, and
Brower\'{}s fixed point theorem for the parabolic equations.

\section{The governing equations of the isothermal P2D model}

We consider the governing equations of the isothermal P2D model for
the variables $u(x,t)$, $v(x;r,t)$, $\phi_{1}(x,t)$ and
$\phi_{2}(x,t)$ presented in Chapters 3 and 4 of \cite{Plet}.
However, to facilitate both the formulation of the numerical method
to integrate these equations and its numerical analysis, it is
convenient to make the changes of variable introduced in \cite{Kro}
and \cite {WXZ}. Thus, in order to make homogeneous the Neumann type
boundary conditions for the potential $\phi_{2}$ one considers the
function $H(x,t)$ given by the expression
\begin{equation*}
H(x,t)=\left\{
\begin{array}{l}
-\displaystyle\frac{(x-L_{n})^{2}I(t)}{2\sigma L_{n}A},\ \ x\in D_{\mathrm{n}%
}, \\
\\
\displaystyle\frac{(x-(L_{n}+\delta ))^{2}I(t)}{2\sigma L_{p}A},\ \ x\in D_{%
\mathrm{p}},%
\end{array}%
\right.
\end{equation*}%
where $I(t)$ denotes the applied current, $A$ is the area of the
plate and $\sigma $ is a positive coefficient defined below, and
replace $\phi _{2}(x,t)$ by $\phi _{2}(x,t)+H(x,t)$; likewise, we
replace the potential $\phi _{1}(x,t)$ by $\phi _{1}(x,t)+\alpha \ln
u(x,t)$, with $\alpha =\alpha(u) =\displaystyle\frac{2RT\kappa(u)
}{F}(t_{+}^{0}-1)$, where $\kappa(u)>0$ denotes the effective
electrolyte phase ionic conductivity; $t_{+}^{0}>0$ is the so called
transfer number, which is assumed to be constant; $R$ is the
universal gas constant and $T$ denotes the absolute temperature
inside the cell, which is assumed to be constant in the isothermal
model; this latter change of variable for $\phi_{1}(x,t)$ simplifies
the expression of the equation for the potential of the electrolyte
phase written in Chapter 4 of \cite{Plet}, making it more manageable
from a computational viewpoint. Another important variable, as we
mentioned above, is the reaction current density $J$. The reaction
rate is coupled to phase
potentials by the Buttler-Volmer kinetic expression.%
\begin{equation}  \label{Ja}
J=J(x,u,v_{s},T,\eta )=\left\{
\begin{array}{l}
a_{s}i_{0}\displaystyle\left( \exp \frac{\alpha _{a}F}{RT}\left( \eta -\frac{%
R_{SEI}}{a_{s}}J\right) \right. \\
\\
-\displaystyle\left. \exp \frac{-\alpha _{c}F}{RT}\left( \eta -\frac{R_{SEI}%
}{a_{s}}J\right) \right) \ \ \text{if \ }x\in D_{n}
\cup D_{p} , \\
\\
0\ \ \text{if \ }x\in D_{s}.%
\end{array}%
\right.
\end{equation}%
In this expression, $v_{s}=v(x;R_{s}(x),t)$ denotes the lithium
concentration on the surface of the active particles; $a_{s}=a_{s}(x)=\displaystyle%
\frac{3\varepsilon _{s}(x)}{R_{s}(x)}$ is the active area per
electrode unit volume; $\varepsilon _{s}(x)$ denotes the volume
fraction of the active material, $\varepsilon _{s}(x)=\varepsilon
_{s}^{-}>0$ for $x \in D_{n}$ and $\varepsilon _{s}(x)=\varepsilon
_{s}^{+}>0$ for $x \in D_{p}$; $\alpha _{a}\in (0,1)$ and $\alpha
_{c}\in (0,1)$ are anodic and cathodic transfer coefficients for an
electron reaction; $R_{SEI}$ represents the solid interface
resistance, usually, $R_{SEI}=0$ in the engineering literature
unless the model also considers aging phenomena of the battery, so
in this paper we take $R_{SEI}=0$.
\begin{equation*}
\eta =\eta (x,\phi _{1},\phi _{2},U)=\left\{
\begin{array}{l}
\phi _{2}(x,t)-\phi _{1}(x,t)-U(x,v_{s})\text{ if\ }x\in
D_{n} \cup D_{p}, \\
\\
0\text{ \textrm{if} }x\in D_{s},%
\end{array}%
\right.
\end{equation*}%
where $U$ stands for the equilibrium potential at the solid
electrolyte interface, which is assumed to be known. $i_{0}$ is the
exchange current density, i.e.,
\begin{equation} \label{i0}
i_{0}=i_{0}(u,v_{s})=ku^{\alpha _{a}}(v_{\max }-v_{s})^{\alpha
_{a}}v_{s}^{\alpha _{c}}\text{ if }x\in D_{n} \cup L_{n},
\end{equation}
here, $v_{\max }$ is the maximum concentration of lithium in the
solid phase, which may have different values in the positive and
negative electrodes, so
\begin{equation*}
v_{\max }=v_{\mathrm{\max }}(x)=\left\{
\begin{array}{l}
v_{\mathrm{\max }}^{-}\text{ \textrm{if} }x\in D_{n}, \\
v_{\mathrm{\max }}^{+}\text{ \textrm{if} }x\in D_{p},%
\end{array}%
\right.
\end{equation*}%
the coefficient $k$ represents the kinetic rate constant,
\begin{equation*}
k=k(x)=\left\{
\begin{array}{l}
k^{-}\text{ \textrm{if} }x\in D_{n}, \\
k^{+}\text{ \textrm{if} }x\in D_{p}.%
\end{array}%
\right.
\end{equation*}
Considering the above mentioned changes of variable and taking the
transfer coefficients $\alpha _{a}$ and $\alpha _{c}$ equal to
$0.5$, as many engineering papers do, the expression for the
reaction current that we use in the paper is
\begin{equation}
J=J(x,u,v_{s},\phi _{1},\phi _{2},\overline{U})=\left\{
\begin{array}{l}
a_{2}(x)i_{0}\sinh \left( \beta \eta \right) ,\ \forall x\in D_{2},
\\
\\
0\ \mathrm{for\ }x\notin D_{2},%
\end{array}%
\right.   \label{F}
\end{equation}%
where $\beta = \displaystyle\frac{F}{2RT}$; $a_{2}(x)=3\varepsilon
_{s}(x)/R_{s}(x)$; and
\begin{equation}
\eta =\phi _{2}-\phi _{1}-\alpha \ln u-\overline{U},  \label{F1}
\end{equation}%
$\overline{U}=\overline{U}(x,t,v_{s})=U(v_{s})-H(x,t)$. Noting that
the boundaries $\partial D_{1}$ and $\partial D_{2}$ of the
domains $D_{1}$ and $D_{2}$ are $\partial D_{1}:=\left\{ 0,L\right\} $ and $%
\partial D_{2}:=\left\{ 0,L_{n},L_{n}+\delta ,L\right\} $, we formulate the
equations of the model as follows.

\bigskip

\textit{Concentration }$u(x,t)$\textit{\ in the electrolyte phase}.

\begin{equation}
\left\{
\begin{array}{l}
\displaystyle\frac{\partial u}{\partial
t}-\displaystyle\frac{\partial}{\partial x}(k_{1}\frac{
\partial u}{\partial x})=a_{1}(x)J\text{ \ \textrm{in} \ }D_{1}\times
(0,T_{\text{\textrm{end}}}), \\
\\
\displaystyle\frac{\partial u}{\partial x}\mid _{\partial D_{1}\times (0,T_{%
\text{\textrm{end}}})}=0,\ u(x,0)=u^{0}(x)\ \text{\textrm{in} }D_{1}.%
\end{array}%
\right.  \label{rf1}
\end{equation}

\textit{Concentration }$v(x;\,r,t)$\textit{\ in the solid phase}. For almost
every $x\in D_{2}$,

\begin{equation}
\left\{
\begin{array}{l}
\displaystyle\frac{\partial v}{\partial t}-\frac{k_{2}}{r^{2}}\frac{%
\partial }{\partial r}\left( r^{2}\frac{\partial v}{\partial r}\right) v=0%
\text{\ \ \textrm{in}\ \ }D_{3}\times (0,T_{\text{\textrm{end}}}), \\
\\
\displaystyle\frac{\partial v}{\partial r}\mid _{r=0}=0,\ k_{\mathrm{2}}%
\displaystyle\frac{\partial v}{\partial r}\mid _{r=R_{\mathrm{s}}(x)}=%
\displaystyle\frac{-J}{a_{2}(x)F},\ v(x;r,0)=v^{0}(x;r)\ \
\text{\textrm{in}
\ }D_{3}.%
\end{array}%
\right.  \label{rf2}
\end{equation}

\textit{Electrolyte potential} $\phi _{\mathrm{1}}(x,t)$.

\begin{equation}
\left\{
\begin{array}{l}
-\displaystyle\frac{\partial }{\partial x}(\kappa (u)\frac{\partial \phi _{%
\mathrm{1}}}{\partial x})=J\text{ \ \textrm{in} \ }D_{1}\times (0,T_{\text{%
end}}), \\
\\
\displaystyle\frac{\partial \phi _{\mathrm{1}}}{\partial x}\mid _{\partial
D_{1}\times (0,T_{\text{end}})}=0, \\
\\
\int_{D_{1}}\phi _{\mathrm{1}}(x,t)dx=0.%
\end{array}%
\right.  \label{rf3}
\end{equation}

\textit{Solid phase potential} $\phi _{\mathrm{2}}(x,t)$.

\begin{equation}
\left\{
\begin{array}{l}
\displaystyle\frac{\partial }{\partial x}\left( \sigma \frac{\partial \phi _{%
\mathrm{2}}}{\partial x}\right) =J+g\text{ \ \textrm{in} \ }D_{2}\times
(0,T_{\text{\textrm{end}}}), \\
\\
\sigma \displaystyle\frac{\partial \phi _{\mathrm{2}}}{\partial x}\mid
_{\partial D_{2}\times (0,T_{\text{\textrm{end}}})}=0, \\
\\
g(x,t)=\left\{
\begin{array}{r}
\displaystyle\frac{-I(t)}{L_{n}A},\ x\in D_{\textrm{n}}, \\
\\
\displaystyle\frac{I(t)}{L_{p}A},\ x\in D_{\textrm{p}},%
\end{array}%
\right.%
\end{array}%
\right.  \label{rf4}
\end{equation}%
where $a_{1}(x)=\displaystyle\frac{1-t_{+}^{0}}{3\varepsilon
_s(x)F}$. In these equations, $k_{1}(x)>0$ and $k_{2}(x)>0$
represent effective diffusion coefficients in the electrolyte and
solid phases respectively, and $\sigma (x) $ denotes the effective
electric conductivity in the solid phase. The functions $a_{1}(x),\
a_{2}(x)\ $, $\sigma (x)$  and $k_{2}(x)$ are considered to be
piecewise positive constant functions in the sense that they have
different constant values in the negative electrode, separator and
positive electrode.

We also have to consider that for $t\in (0,T_{\text{\textrm{end}}})$, $%
J(x,u,v_{\mathrm{s}},\phi _{\mathrm{1}},\phi _{\mathrm{2}})$
satisfies the
algebraic conditions%
\begin{equation}
\left\{
\begin{array}{l}
\displaystyle\int_{D_{1}}Jdx=\int_{D_{2}}Jdx=0, \\
\\
\displaystyle\int_{D_{\textrm{n}}}Jdx=I(t),\ \int_{D_{\textrm{p}}}Jdx=-I(t).%
\end{array}%
\right.  \label{rf5}
\end{equation}
Notice that the first row of algebraic conditions follow directly
from (\ref{rf3}) and the definition of $J$, whereas the second row
conditions translates the boundary conditions of the solid phase
potential. It is worth remarking the conservative properties enjoyed
by both $u(x,t)$
and $v(x;r,t)$; namely, for all $t\in \left[ 0,T_{\mathrm{end}}\right] $%
\begin{equation*}
\int_{D_{1}}u(x,t)dx=\int_{D_{1}}u(x,0)dx,
\end{equation*}%
and%
\begin{equation*}
\int_{D_{2}}\int_{0}^{R_{s}(x)}v(x;r,t)r^{2}drdx=\int_{D_{2}}%
\int_{0}^{R_{s}(x)}v(x;r,0)r^{2}drdx.
\end{equation*}%
These relations are readily obtained by integrating (\ref{rf1}) and (\ref%
{rf2}) and using the corresponding boundary conditions. Moreover, it
can be shown \cite{Kro} that for $t\geq 0$ and $x\in D_{1}$,
$u(x,t)>0$, similarly, for $(x,r)\in D_{3}$, $0<v(x;r,t)<v_{\max }$.

Let $D$ denote a generic open bounded domain in $\mathbb{R}$;
hereafter, the closure of a domain $D$ is denoted $\overline{D}$.
The functional spaces that we use in this paper are the following.
The Sobolev
spaces $H^{m}(D)$, $m$ being a nonnegative integer, when $m=0$, $%
H^{0}(D):=L^{2}(D)$; the Lebesgue spaces $L^{p}(D),\ 1\leq p\leq \infty $;
the spaces of measurable radial functions \cite{Diaz}%
\begin{equation*}
H_{r}^{q}(0,R):=\left\{ v:(0,R)\rightarrow \mathbb{R}\text{:\textrm{\ }}%
\left\Vert v\right\Vert
_{H_{r}^{q}(0,R)}^{2}=\sum_{j=0}^{q}\int_{0}^{R}\left( \frac{d^{j}v}{dr^{j}}%
\right) ^{2}r^{2}dr<\infty \right\} ,
\end{equation*}%
$q$ being a nonnegative integer, when $q=0$ we set $%
H_{r}^{0}(0,R):=L_{r}^{2}(0,R)$; also, for $p$ being a nonnegative
integer,
the normed spaces of measurable functions%
\begin{equation*}
H^{p}(D_{2};H_{r}^{q}(0,R_{s}(\cdot ))):=\left\{ v:D_{2}\rightarrow
H_{r}^{q}(0,R_{s}(\cdot )):\left\Vert v\right\Vert
_{H^{p}(D_{2};H_{r}^{q}(0,R_{s}(\cdot )))}<\infty \right\} ,
\end{equation*}%
where $\left\Vert v\right\Vert _{H^{p}(D_{2};H_{r}^{q}(0,R_{s}(\cdot
)))}^{2}=$ $\displaystyle\sum_{j=0}^{p}\int_{D_{2}}\left\Vert \frac{%
\partial ^{j}v(x;\cdot )}{\partial x^{j}}\right\Vert
_{H_{r}^{q}(0,R_{s}(x))}^{2}dx$; and the spaces%
\begin{equation*}
H_{\ \ r}^{p,q}(D_{2}\times (0,R_{s}(\cdot
)))=H^{p}(D_{2},L_{r}^{2}(0,R_{s}(\cdot )))\cap
L^{2}(D_{2},H_{r}^{q}(0,R_{s}(\cdot )))
\end{equation*}%
with norm%
\begin{equation*}
\left\Vert u\right\Vert _{H_{\ \ r}^{p,q}(D_{2}\times (0,R_{s}(\cdot
)))}^{2}:=\left\Vert u\right\Vert
_{H^{p}(D_{2},L_{r}^{2}(0,R_{s}(\cdot )))}^{2}+\left\Vert
u\right\Vert _{L^{2}(D_{2},H_{r}^{q}(0,R_{s}(\cdot )))}^{2}.
\end{equation*}%
Notice that $H_{\ \ r}^{0,0}(D_{2}\times (0,R_{s}(\cdot
)))=L^{2}(D_{2},L_{r}^{2}(0,R_{s}(\cdot ))).$

Since the variables of the model depend on time, then we also introduce the
normed spaces $L^{p}(0,t;X)$, where $1\leq p\leq \infty $, and $(X,\
\left\Vert \cdot \right\Vert _{X})$ being a real Banach space.%
\begin{equation*}
L^{p}(0,t;X):=\left\{ v:(0,t)\rightarrow X\text{ \textrm{strongly measurable
such that\ \ }}\left\Vert v\right\Vert _{L^{p}(0,t;X)}<\infty \right\} ,
\end{equation*}%
with $\left\Vert v\right\Vert _{L^{p}(0,t;X)}=\displaystyle\left(
\int_{0}^{t}\left\Vert v(\tau )\right\Vert _{X}^{p}d\tau \right)
^{1/p}$ when $1\leq p<\infty $, and for $p=\infty $, $\left\Vert
v\right\Vert _{L^{\infty }(0,t;X)}=ess\ \sup_{0<\tau <t}\left\Vert
v(\tau )\right\Vert _{X}$. Other spaces used in the paper are
$W(D_{1}):=\left\{ v\in
H^{1}(D_{1}):\displaystyle\int_{D_{1}}vdx=0\right\} $, which is a
closed subspace of $H^{1}(D_{1})$ where the potential $\phi
_{1}(x,t)$ is calculated, and the space of $q$ times continuously
differentiable functions defined on $D$, $C^{q}(D)$, when $q=0$,
$C^{0}(D):=C(D)$.

Next, we introduce the following regularity assumptions on the data and the
molar flux $J$ \cite{Diaz}.

\textbf{A1)}%
\begin{equation*}
u^{0}\in H^{1}(D),\ u^{0}>0,\ v^{0}\in C(\overline{D}_{3}),\ 0<v^{0}<v_{\max
},\ I(t)\in C_{part}([0,T^{\ast }]),0<T_{end}\leq T^{\ast }<\infty ,
\end{equation*}%
where $C_{part}$ denotes the set of piecewise continuous functions, i.e.,%
\begin{equation*}
C_{\mathrm{part}}([a,b])=\left\{ g:[a,b]\rightarrow \mathbb{R}:\exists
a=t_{0}<t_{1}<\cdots t_{N}=b\text{ such that }g\in C\left(
[t_{i-1},t_{i}]\right) \right\} .
\end{equation*}

\textbf{A2)} For $0<a<c<+\infty $, $k_{0}$ and $\sigma _{0}$
positive constants,
\begin{equation*}
k_{1}\in L^{\infty }(D_{1}),\ k_{1}\geq k_{0}>0,\ k_{2}\in \lbrack
a,c],\ \kappa \in C^{2}\left( (0,+\infty )\right) ,\ \sigma \in
L^{\infty }(D_{2}),\ \kappa \geq \kappa _{0}>0,\ \sigma \geq \sigma
_{0}>0.
\end{equation*}%
\ \ \ \ \   Moreover, $\underline{k_{1}}:=\inf_{x\in
{D_{1}}}{k_{1}(x)}$ and $\underline{k_{2}}:=\min_{x\in
{D_{2}}}{k_{2}(x)}$.

\bigskip

\textbf{A3)} For all $(x,u,v_{s},\eta )\in D_{2}\times (0,+\infty )\times
(0,v_{s,\mathrm{\max }})\times \mathbb{R}$,%
\begin{equation*}
J\in C^{2}(D_{2}\times (0,+\infty )\times (0,v_{s,\mathrm{\max }})\times
\mathbb{R}),\ \displaystyle\frac{\partial J}{\partial \eta }>0.
\end{equation*}

A weak formulation to (\ref{rf1})-(\ref{rf4}) is the following. Find

\begin{equation*}
\left\{
\begin{array}{l}
u\in {L}^{2}\left( 0,T_{\mathrm{end}};\ H^{1}(D_{1}\right) ),\ \displaystyle%
\frac{du}{dt}\in L^{2}(0,T_{\mathrm{end}};\ H^{1}(D_{1}))^{\ast }, \\
\\
v{\in L}^{2}\left( 0,T_{\mathrm{end}};L^{2}(D_{2,}\
H_{r}^{1}(0,R_{s}(\cdot
))\right) ),\ \displaystyle\frac{dv}{dt}\in {L}^{2}\left( 0,T_{\mathrm{end}%
};L^{2}(D_{2,}\ H_{r}^{1}(0,R_{s}(\cdot ))\right) )^{\ast }, \\
\\
\phi _{\mathrm{1}}\in L^{2}(0,T_{\mathrm{end}};\ W(D_{1}))\text{
\textrm{and}
}\phi _{\mathrm{2}}\in L^{2}(0,T_{\mathrm{end}};\ H^{1}(D_{2})),%
\end{array}%
\right.
\end{equation*}%
such that%
\begin{equation}
\int_{D_{1}}\frac{\partial {u}}{\partial {t}}wdx+\int_{D_{1}}k_{1}\frac{%
\partial u}{\partial x}\frac{dw}{dx}dx=\int_{D_{1}}a_{1}Jwdx\ \forall w\in
H^{1}(D_{1});  \label{w1}
\end{equation}%
for a.e.$x\in D_{2}$ and for all $w\in H_{r}^{1}(0,R_{s}(x))$
radially symmetric%
\begin{equation}
\int_{0}^{R_{s}(x)}\frac{\partial {v}}{\partial {t}}wr^{2}dr+%
\int_{0}^{R_{s}(x)}k_{2}\frac{\partial v}{\partial r}\frac{\partial w}{%
\partial r}r^{2}dr=\frac{-R_{\mathrm{s}}^{2}(x)Jw(R_{s}(x))}{a_{2}(x)F};  \label{w2}
\end{equation}

\begin{equation}
\int_{D_{1}}\kappa (u)\frac{\partial \phi _{1}}{\partial x}\frac{dw}{dx}%
dx=\int_{D_{1}}Jwdx\text{ \ }\forall w\in H^{1}(D_{1});  \label{w3}
\end{equation}%
and%
\begin{equation}
\int_{D_{2}}\sigma \frac{\partial \phi _{2}}{\partial x}\frac{dw}{dx}%
dx=-\int_{D_{2}}\left( J+g\right) wdx,\forall w\in H^{1}(D_{2}),
\label{w4}
\end{equation}%
where $L^{2}(0,T_{\mathrm{end}};\ H^{1}(D_{1}))^{\ast }$ and
${L}^{2}\left( 0,T_{\mathrm{end}};L^{2}(D_{2,}\
H_{r}^{1}(0,R_{s}(\cdot ))\right) )^{\ast }$ denote the respective
dual spaces of\ $L^{2}(0,T_{\mathrm{end}};\ H^{1}(D_{1}))$ and
${L}^{2}\left( 0,T_{\mathrm{end}};L^{2}(D_{2,}\
H_{r}^{1}(0,R_{s}(\cdot ))\right) )$.
\begin{remark}
\label{remark1} Following the arguments of \cite{Diaz}, where its
non-isothermal P2D model includes an additional time dependent non
linear ordinary equation for the bulk temperature $T(t)$, one can
formulate an alternative definition
of the weak solution to (\ref{rf1}%
)-(\ref{rf4}) based on its Definition 2.7 and prove, under the
assumptions \textbf{A1}-\textbf{A3} and for a
partition $t_{0}<t_{1}<\cdots <t_{N}$ of $[0,T_{%
\mathrm{end}}]$, $T_{\mathrm{end}}$ being small enough, that there
is a unique
weak solution $\left( u,v,\phi _{1},\phi _{2}\right) $ in each interval $%
[t_{n},t_{n+1}]$, such that $\left( u,v,\phi _{1},\phi _{2}\right)
\in C([t_{n},t_{n+1}];K_{Z})$, where $K_{Z}:=H^{1}(D_{1})\times
L^{2}(D_{2,}\ H_{r}^{1}(0,R_{s}(\cdot )))\times W(D_{1})\times
H^{1}(D_{2})$. $\left( u(t_{n}),v(t_{n})\right)$ being the initial
condition in such an interval. Also, Kr$\ddot{\text{o}}$ner
\cite{Kro} proves a local existence and uniqueness theorem for the
weak solution of the isothermal P2D model under less general
assumptions than in \cite{Diaz}.
\end{remark}

\section{The semidiscrete finite element formulation of the isothermal P2D
model}

We use $H^{1}$-conforming linear finite elements ($P_{1}-$finite
elements) for the space approximation of the variables $u(x,t)$,
$\phi _{\mathrm{1}}(x,t)$ and $\phi _{\mathrm{2}}(x,t)$; however,
$v(x;r,t)$ is approximated by nonconforming $P_{0}-$finite elements
in the $x-$coordinate and $H^{1}$-conforming $P_{1}-$finite elements
in the $r- $coordinate. The family of meshes $D_{1h}$ constructed on
the domain $D_{1}$ includes the points $x=0$, $x=L_{n}$,
$x=L_{n}+\delta $ and $x=L$ as mesh points; since these points are
also boundary points of $D_{2}$, then they are also considered as
mesh points in the family of meshes $D_{2h}$. Figure 2 illustrates
the families of meshes that we are going to describe next. Noting
that $D_{2}\subset D_{1}$, we choose the family of meshes $D_{2h}$
as a subset of $D_{1h}$. Let $NE_{1}$ and $%
NE_{2}$ be the number of elements of $D_{1h}$ and $D_{2h}$
respectively, and let $M_{1}$ and $M_{2}$ be the number of mesh
points of such meshes, then, for $i=1,2$, we have that
\begin{equation*}
D_{ih}=\left\{ e_{m}\right\} _{m=1}^{NE_{i}}\ \text{\textrm{and}}\ \overline{%
D}_{i}=\cup _{m=1}^{NE_{i}}e_{m},
\end{equation*}%
where the $m$th element, $e_{m}:=\left\{ x:x_{1}^{m}\leq x\leq x_{2}^{m}\right\} $, and $%
h_{m}:=x_{2}^{m}-x_{1}^{m}$ is the length of the element $e_{m}$;
the points $\ x_{1}^{m}$ $x_{2}^{m}$ are denoted element nodes. We
set $h=\max_{m}h_{m}$, and $ \gamma =h^{-1}\min_{m}h_{m}$. The
parameter $\gamma $ is a measure of the uniformity of the meshes.
The collection of all the element nodes defines
the set of nodes, $\left\{ x_{l}\right\} _{l=1}^{M_{i}}$, of the mesh $%
D_{ih} $. To construct the family of meshes $\widehat{D}_{3h\Delta
r}$ on $D_{3}$, we recall that $D_{3}=\cup _{x\in D_{2}}\left\{
x\right\} \times (0,R_{s}(x))$, where $R_{s}(x)$ is the radius of
the solid spherical particle associated with the point $\{x\}$.
Thus, for each mesh point $\left\{ x_{l}\right\} \in D_{2h}$ we
define the radial vertical domain $D_{r}^{(l)}:=\left\{ r\in
\mathbb{R}
:0<r<R_{s}(x_{l})\right\} $, which represents the spherical particle at $%
x_{l}$, and let%
\begin{equation*}
D_{\Delta r}^{(l)}=\left\{ e_{k}^{(l)}\right\} _{k=1}^{NE^{(l)}}\text{
\textrm{such that }}\overline{D}_{r}^{(l)}=\cup _{k=1}^{NE^{(l)}}e_{k}^{(l)},
\end{equation*}%
where $NE^{(l)}$ denotes the number of elements in the interval $%
[0,R_{s}(x_{l})]$ and $\Delta r_{k}^{(l)}$ is the width of the
element $e_{k}^{(l)}:=\{r:r_{1}^{(l)k}\leq r\leq r_{2}^{(l)k}\}$, we
set $\Delta r=\max_{l}(\max_{k}\Delta r_{k}^{(l)})$ and $\gamma
_{r}=\Delta r^{-1}\min_{l}(\min_{k}\Delta r_{k}^{(l)}).$ The set of
mesh points in each mesh $D_{\Delta r}^{(l)}$ is denoted $\left\{
r_{j}^{(l)}\right\} _{j=1}^{M^{(l)}}$. 
Furthermore, let $\left\{ \widehat{e}_{l}\right\} _{l=1}^{M_{2}}$ be
the collection of nonconforming elements of the mesh
$\widehat{D}_{2h}$ which are associated with the nodes $\left\{
x_{l}\right\} $, they are defined as
follows: if $\left\{ x_{l}\right\} $ is not a boundary point, then%
\begin{equation*}
\widehat{e}_{l}:=\left\{ x\in D_{2}:x_{l}-\frac{h_{l-1}}{2}\leq x<x_{l}+%
\frac{h_{l}}{2}\right\} ;
\end{equation*}%
on the contrary, if $\left\{ x_{l}\right\} $ is a left boundary point, then%
\begin{equation*}
\widehat{e}_{l}:=\left\{ x\in \overline{D}_{2}:x_{l}\leq x<x_{l}+\frac{%
h_{l}}{2}\right\} ,
\end{equation*}%
and if $\left\{ x_{l}\right\} $ is a right boundary point, then%
\begin{equation*}
\widehat{e}_{l}:=\left\{ x\in
\overline{D}_{2}:x_{l}-\frac{h_{l-1}}{2}\leq x\leq x_{l}\right\} .
\end{equation*}%
We define the meshes $\widehat{D}_{3h\Delta r}$ as
\begin{equation*}
\widehat{D}_{3h\Delta r}:=\left\{ \widehat{e}_{l}\times
\overline{D}_{\Delta r}^{(l)}\right\} _{l=1}^{M_{2}}\
\text{\textrm{such that }}\overline{D}_{3}:=\cup
_{l=1}^{M_{2}}\widehat{e}_{l}\times \overline{D}_{\Delta r}^{(l)}.
\end{equation*}
\begin{figure}[th!]
\begin{center}
\includegraphics[height=10cm]{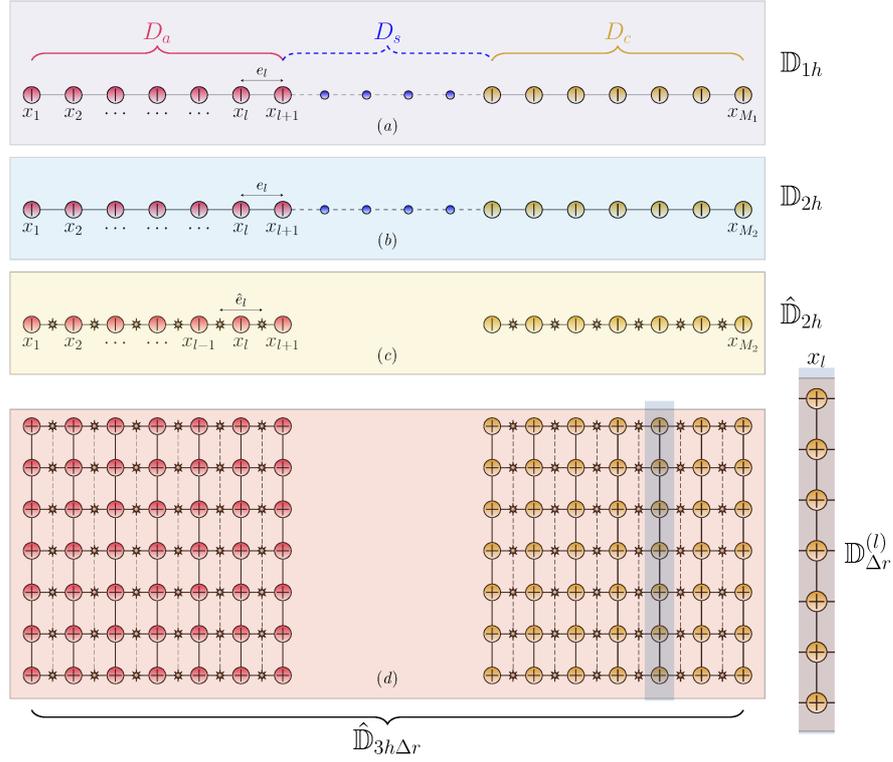}
\end{center}
\caption{\textit{Panel (a)}: the mesh for the domain $D_{1}$, which
includes the negative electrode (anode in the figure) $D_{a}$, the
separator $D_{s}$ and the positive electrode (cathode in the figure)
$D_{c}$. \textit{%
Panel (b)}: the mesh for the domain $D_{2}=D_{a}\cup D_{c}$.
\textit{Panel (c)}: the mesh of nonconforming elements for the domain $D_{2}$. \textit{%
Panel (d)}: the mesh for the domain $D_{3}$.} \label{figure2}
\end{figure}

The families of conforming linear finite element spaces associated with
these meshes are the following. For $i=1,2,$%
\begin{equation*}
V_{h}^{(1)}(\overline{D}_{i}):=\left\{ v_{h}\in C(\overline{D}%
_{i}):\forall e_{m}\in D_{ih},\ v_{h}\mid _{e_{m}}\in P_{1}(e_{m})\right\} ,
\end{equation*}%
where $P_{1}(e_{m})$ denotes the set of linear polynomials defined on $e_{m}$%
.\ Let $\left\{ \psi _{l}(x)\right\} _{l=1}^{M_{i}}$ be the set of nodal
basis functions for the space $V_{h}^{(1)}(\overline{D}_{i})$, then any
function $v_{h}\in V_{h}^{(1)}(\overline{D}_{i})$ can be written as%
\begin{equation*}
v_{h}(x)=\sum_{l=1}^{M_{i}}V_{l}\psi _{l}(x),\ \mathrm{where\ }%
V_{l}=v_{h}(x_{l})\text{.}
\end{equation*}%
Note that $V_{h}^{(1)}(\overline{D}_{i})\subset H^{1}(D_{i})$. The
nonconforming finite element space associated with the mesh
$\widehat{D}_{2h}$ is defined as
\begin{equation*}
V_{h}^{(0)}(\overline{D}_{2}):=\left\{ v_{h}\in L^{2}(D_{2}):\forall
\widehat{e}_{l}\in \widehat{D}_{2h},\ v_{h}\mid
_{\widehat{e}_{l}}\in P_{0}(\widehat{e}_{l})\right\} ,
\end{equation*}
where $P_{0}(\widehat{e}_{l})$ is the set of polynomials of degree
zero defined on $\widehat{e}_{l}$. Let $\left\{ \chi _{l}(x)\right\}
_{l=1}^{M_{2}}$ be the set of nodal basis functions for $V_{h}^{(0)}(%
\overline{D}_{2})$,%
\begin{equation*}
\chi _{l}(x)=\left\{
\begin{array}{r}
1\text{ \ \textrm{if} \ }x\in \widehat{e}_{l}, \\
\\
0\ \ \text{\textrm{otherwise}},%
\end{array}%
\right.
\end{equation*}%
then any function $v_{h}(x)\in V_{h}^{(0)}(\overline{D}_{2})$ is expressed as%
\begin{equation*}
v_{h}(x)=\sum_{l=1}^{M_{2}}V_{l}\chi _{l}(x),\ \mathrm{where\ }%
V_{l}=v_{h}(x_{l}).
\end{equation*}%
It is worth remarking that for $1\leq p<\infty$, the $L^{p}$-norm of $%
v_{h}(x)\in V_{h}^{(0)}(\overline{D}_{2})$ is given as $\left\Vert
v_{h}\right\Vert _{L^{p}(D_{2})}^{p}=\sum_{l=1}^{M_{2}}\widehat{h}
_{l}V_{l}^{p} $, where $\widehat{h}_{l}$ denotes the length of the
element $\widehat{e}_{l}$, and the $L^{2}$-inner product of
$v_{h},w_{h} \in V_{h}^{(0)}(\overline{D}_{2})$,
$\displaystyle\int_{D_{2}}v_{h}w_{h}dx=\sum_{l=1}^{M_{2}}
\widehat{h}_{l}V_{l}W_{l}$. Next, we introduce the finite element
space $V_{\Delta
r}^{(1)}(\overline{D}_{r}^{(l)})$. For $1\leq l\leq M_{2}$,%
\begin{equation*}
V_{\Delta r}^{(1)}(\overline{D}_{r}^{(l)}):=\left\{ v_{\Delta r}\in C(%
\overline{D}_{r}^{(l)}):\forall e_{k}^{(l)}\in D_{\Delta r}^{(l)},\
v_{\Delta r}^{(l)}(r)\mid _{e_{k}^{(l)}}\in P_{1}(e_{k}^{(l)})\right\} .
\end{equation*}%
So, if $\left\{ \alpha _{j}^{(l)}(r)\right\} _{j=1}^{M^{(l)}}$
denotes the set of nodal basis of $V_{\Delta
r}(\overline{D}_{r}^{(l)})$, any function $v_{\Delta r}^{(l)}\in
V_{\Delta r}^{(1)}(\overline{D}_{r}^{(l)})\subset
H_{r}^{1}(0,R_{s}(x_{l}))$ can be written as%
\begin{equation*}
v_{\Delta r}^{(l)}(r)=\sum_{j=1}^{M^{(l)}}V_{j}^{(l)}\alpha
_{j}^{(l)}(r),\ \mathrm{where\
}V_{j}^{(l)}=v_{h}^{(l)}(r_{j})\text{.}
\end{equation*}%
Regarding the meshes $\widehat{D}_{3h\Delta r}$, we define the finite element space $%
V_{h\Delta r}(\overline{D}_{3})$ as follows. For $1\leq l\leq M_{2}$ and $%
1\leq k\leq NE^{(l)}$%
\begin{equation*}
V_{h\Delta r}(\overline{D}_{3}):=\displaystyle\left\{ {v}_{h\Delta
r}{\in H}
_{\ \ r}^{0,1}{(D_{2}\times (0,R_{s})(\cdot )):v_{h\Delta r}(x;r)\mid _{%
\widehat{e}_{l}\times e_{k}^{(l)}}\in P_{0}(\widehat{e}_{l})\otimes
P_{1}(e_{k}^{(l)})}\right\},
\end{equation*}%
noting that when $x\in \widehat {e}_{l}$, $R_{s}(x)=R_{s}(x_{l})$.
Hence, any function $v_{h\Delta r}(x;r)\in V_{h\Delta
r}(\overline{D}_{3})$
is of the form%
\begin{equation*}
v_{h\Delta r}(x;r)=\sum_{l=1}^{M_{2}}\sum_{j=1}^{M^{(l)}}V_{lj}\chi
_{l}(x)\alpha _{j}^{(l)}(r),\ \mathrm{where\ }V_{lj}=v_{h\Delta
r}(~x_{l},r_{j}),
\end{equation*}%
or equivalently, using the notation $v_{\Delta r}^{(l)}(r)$ to
denote $v_{h\Delta
r}(x_{l};r)$, we can write%
\begin{equation} \label{vhr}
v_{h\Delta r}(x;r)=\sum_{l=1}^{M_{2}}v_{\Delta r}^{(l)}(r)\chi _{l}(x).
\end{equation}.
The function $v_{sh}(x):=v_{h\Delta r}(x;R_{s}(x))$ is given by the
expression%
\begin{equation*}
v_{sh}(x;R_{s}(x))=\sum_{l=1}^{M_{2}}V_{lM^{(l)}}\chi _{l}(x),
\end{equation*}%
so that $v_{sh}(x)\in V_{h}^{(0)}(\overline{D}_{2})$. We calculate
$\phi _{1h}(x,t)$, which is the approximation to $\phi_{1}(x,t)$, in
the finite
dimensional space%
\begin{equation*}
W_{h}(\overline{D}_{1}):=\left\{ v_{h}\in V_{h}^{(1)}(\overline{D}%
_{1}):\int_{D_{1}}v_{h}dx=0\right\} ,\ W_{h}(\overline{D}_{1})\subset
W(D_{1}).
\end{equation*}

Thus, the finite element formulation is as follows. For all $t\in (0,T_{%
\mathrm{end}})$, the semi-discrete approximation $(u_{h}(t),v_{h\Delta
r}(t),\phi _{1h}(t),\phi _{2h}(t))\in V_{h}^{(1)}(\overline{D}_{1})\times
V_{h\Delta r}(\overline{D}_{3})\times W_{h}(\overline{D}_{1})\times
V_{h}^{(1)}(\overline{D}_{2}),$ $(u_{h}^{0},v_{h}^{0})\in V_{h}^{(1)}(%
\overline{D}_{1})\times V_{h\Delta r}(\overline{D}_{3})$, is solution to the
following system of equations.%
\begin{equation}
\int_{D_{1}}\frac{\partial u{_{h}}}{\partial {t}}w_{h}dx+\int_{D_{1}}k_{1}%
\frac{\partial u_{h}}{\partial x}\frac{dw_{h}}{dx}dx=%
\int_{D_{1}}a_{1}J_{h}w_{h}dx\ \forall w_{h}\in V_{h}^{(1)}(\overline{D}%
_{1}).  \label{num1}
\end{equation}

\begin{equation}
\left\{
\begin{array}{l}
\displaystyle\int_{D_{2}}\int_{0}^{R_{s}(x)}\displaystyle\frac{\partial v{%
_{h\Delta r}}}{\partial {t}}w_{h\Delta r}r^{2}drdx+\displaystyle%
\int_{D_{2}}\int_{0}^{R_{s}(x)}k_{2}\displaystyle\frac{\partial v{_{h\Delta
r}}}{\partial {r}}\frac{\partial w_{h\Delta r}}{\partial r}r^{2}drdx \\
\\
=-\displaystyle\int_{D_{2}}\frac{R_{s}^{2}(x)J_{h}w_{h\Delta
r}(x,R_{s}(x))}{a_{2}(x)F}dx\ \ \forall w_{h\Delta r}\in V_{h\Delta
r}(\overline{D}_{3}).
\end{array}%
\right.  \label{num2}
\end{equation}

\begin{equation}
\int_{D_{1}}\kappa (u_{h})\frac{\partial \phi _{1h}}{\partial x}\frac{dw_{h}%
}{dx}dx=\int_{D_{1}}J_{h}w_{h}dx\ \ \forall w_{h}\in V_{h}^{(1)}(\overline{D}%
_{1}).  \label{num3}
\end{equation}

\bigskip

\begin{equation}
\int_{D_{2}}\sigma \frac{\partial \phi _{2h}}{\partial x}\frac{dw_{h}}{dx}%
dx=-\int_{D_{2}}\left( J_{h}+g\right) w_{h}dx\ \ \forall w_{h}\in
V_{h}^{(1)}(\overline{D}_{2}).  \label{num4}
\end{equation}

\bigskip

\begin{equation}
\int_{D_{1}}J_{h}dx=\int_{D_{2}}J_{h}dx=0, \text{ with
}\int_{D_{\mathrm{n}}}J_{h}dx=I(t)=- \int_{D_{\textrm{p}}}J_{h}dx.
\label{num5}
\end{equation}%
In this system,

\begin{equation}
\begin{array}{l}
J_{h}:=J(x,u_{_{h}},v_{sh},\eta _{h})=a_{2}(x)i_{0h}\sinh \left(
\beta \eta _{h}\right), \
i_{0h}:=i_{0}(u_{h},v_{sh}) \\
\\
\eta _{h}:=\phi _{1h}-\phi _{2h}-\alpha_{h} \ln
u_{h}-\overline{U}_{h}(v_{sh}),\ \alpha_{h} =\alpha(u_{h}).
\end{array}
\label{num6}
\end{equation}%
Note that $J_{h}$ also depends on $t$ through $u_{h}$, $v_{sh}$ and $\eta
_{h}$.

\begin{remark}
\label{remark2} We must note, see (\ref{vhr}), that $v_{h\Delta r}$
and $w_{h\Delta r}$ are elementwise constant functions in the $x$
direction, and $J_{h}$ is a piecewise continuous function in $x$ for
which it makes sense to consider the approximation,
$I_{h}^{(0)}J_{h}\in V_{h}^{(0)}(\overline{D}_{2})$, see in
Subsection 4.1 the definition of the interpolant $I_{h}^{(0)}$.
Then, approximating $J_{h}$ by $I_{h}^{(0)}J_{h}$ one readily shows,
by performing the integral on $D_{2}$,  that (\ref{num2}) can be
recast as follows: for all mesh-point $\left\{ x_{l}\right\} \in
D_{2h}$, calculate $v_{\Delta
r}^{(l)}(r,t)$ $\in V_{\Delta r}^{(1)}(\overline{D}_{r}^{(l)})$ such that%
\begin{equation} \label{num2_1}
\begin{array}{r}
\displaystyle\int_{0}^{R_{s}(x_{l})}\frac{\partial v_{\Delta r}^{(l)}(r,t)}{%
\partial t}w_{\Delta r}^{(l)}(r)r^{2}dr+\int_{0}^{R_{s}(x_{l})}\displaystyle %
k_{2}\frac{\partial v_{\Delta r}^{(l)}(r,t)}{\partial
r}\frac{\partial
w_{\Delta r}^{(l)}(r)}{\partial r}r^{2}dr \\
\\
=-R_{s}^{2}(x_{l})(a_{2}(x_{l})F)^{-1}J_{h}\mid _{x_{l}}w_{\Delta
r}^{(l)}(R_{s}(x_{l}))\ \ \forall w_{\Delta r}^{(l)}\in V_{\Delta r}^{(1)}(%
\overline{D}_{r}^{(l)}).%
\end{array}%
\end{equation}
Once $v_{\Delta r}^{(l)}(r,t)$ is known, one calculates $v_{h\Delta
r}$ by the expression (\ref{vhr}).
\end{remark}

Notice that this equation is the finite element approximation of (\ref{w2})
for $w\in H_{r}^{1}(0,R_{s}(x_{l}))$, see \cite{BG}.

Based on Remark \ref{remark1} and since $H^{1}(D_{i})\hookrightarrow
C(\overline{D}_{i})$, we introduce the following spaces which are
used
in the error analysis and in the application of the fixed point theorems in Section 4.%
\begin{equation*}
\begin{array}{l}
S_{P}:=\left\{ u\in C_{\mathrm{part}}(\overline{D}_{1}\times \lbrack 0,T_{%
\mathrm{end}}]):\displaystyle\frac{1}{P}\leq u\leq P\right\} , \\
\\
S_{Q}:=\left\{ w\in C_{\mathrm{part}}(\overline{D}_{2}\times \lbrack 0,T_{%
\mathrm{end}}]):\displaystyle\frac{1}{Q}\leq w\leq 1-\frac{1}{Q}\right\} , \\
\\
S_{Q}^{\ast }:=\left\{ w\in C_{\mathrm{part}}^{\ast
}(\overline{D}_{2}\times
\lbrack 0,T_{\mathrm{end}}]):\frac{1}{Q}\leq w\leq 1-\frac{1}{Q}\right\} ,%
\end{array}%
\end{equation*}%
where $P$ and $Q$ are constants sufficiently large; for $i=1,2$, $C_{\mathrm{part}%
}(\overline{D}_{i}\times [0,T_{\mathrm{end}}])$ denotes the set of
piecewise
continuous functions in time and continuous in space and $C_{\mathrm{part}%
}^{\ast }(\overline{D}_{2}\times [0,T_{\mathrm{end}}])$ denotes the
set of piecewise continuous functions in both time and space.
$S_{P}$ is the candidate pool for the concentration $u$ and its
approximate $u_{h}$, whereas $S_{Q}$ and $S_{Q}^{\ast}$ play the
same role for the concentrations $\frac{v_{s}}{v_{\max }}$ and
$\frac{v_{sh}}{v_{\max }}$ respectively. So, we make the following
assumption.

\textbf{A4)\ }There exist constants $P,\ Q$ and $K$ sufficiently
large such that that for almost
every $t\in \lbrack 0,T_{\mathrm{end}}]$ the following bounds hold:%
\begin{equation}
\frac{1}{P}\leq u(t),\ u_{h}(t)\leq P,\text{ \ }\ \frac{1}{Q}\leq
\frac{v_{s}(t)}{v_{\max }},\ \frac{v_{sh}(t)}{v_{\max }}\leq \left(
1-\frac{1}{Q}\right) , \label{estK1}
\end{equation}%
and for $i=1,2$,

\begin{equation}
\left\Vert \frac{\partial \phi _{1}(t)}{\partial x}\right\Vert _{L^{\infty
}(D_{1})},\ \left\Vert \phi _{i}(t)\right\Vert _{H^{1}(D_{i})},\ \left\Vert
\phi _{i}(t)\right\Vert _{L^{\infty }(D_{i})},\ \left\Vert \phi
_{ih}(t)\right\Vert _{L^{\infty }(D_{i})}\leq K.  \label{estK2}
\end{equation}


\section{Error analysis for the semidiscrete problem}

We present in this section the error analysis for the semidiscrete
potentials and concentrations. Since the development of such an
analysis is long, we have split its presentation in a sequence of
three subsections. In the first one, we introduce some auxiliary
results needed for the error analysis. The second subsection deals
with the error estimate for the potentials. Observing that the P2D
model is a nonlinear coupled system of equations, then the error for
the potentials depends on the error estimates for the
concentrations, the analysis of which is carried out in the last
subsection.

\subsection{Auxiliary results}

It is well known \cite{Ci} that for the finite element spaces $V_{h}^{(p)}(\overline{D}%
_{i})$ ($i=1,2; p=0,1$) the following approximation property holds. For $%
1\leq s\leq p+1$,%
\begin{equation}
\inf_{v_{h}\in V_{h}^{(p)}(D_{i})}\left\{ \left\Vert
v-v_{h}\right\Vert _{L^{2}(D_{i})}+h\left\Vert \frac{d^{p}\left(
v-v_{h}\right) }{dx^{p}}\right\Vert _{L^{2}(D_{i})}\right\} \leq
Ch^{s}\left\vert v\right\vert _{H^{s}(D_{i})}, \label{pre1}
\end{equation}%
where it should be understood that for $p=0$,
$\displaystyle\frac{d^{p}\left( v-v_{h}\right) }{dx^{p}}=v-v_{h}$.
For symmetric radial functions defined in the interval $[0,R]$, let
$V_{\Delta r}^{(1)}[0,R]$ be a linear finite element space where we
approximate such functions, one can prove, following the approach
used to
prove Lemmas 1 and 2 in \cite{ET}, that when $w\in H_{r}^{2}(0,R)$,%
\begin{equation}
\inf_{w_{\Delta r}\in V_{\Delta r}^{(1)}[0,R]}\left\{ \left\Vert
w-w_{\Delta r}\right\Vert _{L_{r}^{2}(0,R)}+\Delta r\left\Vert
\frac{d\left( w-w_{\Delta r}\right) }{dr}\right\Vert
_{L_{r}^{2}(0,R)}\right\} \leq C\Delta r^{2}\left\vert w\right\vert
_{H_{r}^{2}(0,R)}.  \label{pre2}
\end{equation}%
We consider the interpolants $I_{h}^{(1)}:C(\overline{D}_{i})\rightarrow $ $%
V_{h}^{(1)}(\overline{D}_{i})$, $I_{h}^{(0)}:C(\overline{D}%
_{2})\rightarrow V_{h}^{(0)}(\overline{D}_{2}),$ and the elliptic projection
$P_{1}:H^{1}(D_{1})\rightarrow V_{h}^{(1)}(D_{1})$ such that for $u\in
H^{1}(D_{1})$%
\begin{equation}
\int_{D_{1}}\left(
k_{1}\frac{d(P_{1}u-u)}{dx}\frac{du_{h}}{dx}+\lambda
(P_{1}u-u)u_{h}\right) dx=0\ \ \forall u_{h}\in V_{h}^{(1)}(D_{1}),
\label{pre3}
\end{equation}%
where $\lambda >0$ is a constant; the error analysis for elliptic
problems suggests that a good choice is $\lambda
=\underline{k_{1}}$. By virtue of (\ref{pre1}) it follows that\
there exists a constant $C$ independent of $h$ such that
\begin{equation}
\left\Vert v-I_{h}^{(0)}v\right\Vert _{L^{2}(D_{2})}\leq Ch\left\Vert
v\right\Vert _{H^{1}(D_{2})};  \label{pre4}
\end{equation}%
for $1\leq m\leq 2,\ 0\leq l\leq 1,$

\begin{equation}
\left\Vert v-I_{h}^{(1)}v\right\Vert _{H^{l}(D_{i})}\leq
Ch^{m-l}\left\Vert v\right\Vert _{H^{m}(D_{i})};  \label{pre4.1}
\end{equation}%
and from the well known error analysis for elliptic problems
\cite{Ci}
\begin{equation}
\left\Vert u-P_{1}u\right\Vert _{H^{l}(D_{1})}\leq
Ch^{m-l}\left\Vert u\right\Vert _{H^{m}(D_{1})}.  \label{pre5}
\end{equation}%
Likewise, for symmetric radial functions $v\in H_{r}^{q}(0,R)$, we define
the elliptic projector $P_{1}^{r}:H_{r}^{1}(0,R)\rightarrow V_{\Delta
r}^{(1)}[0,R]$ as the solution of the problem%
\begin{equation}
\int_{0}^{R}\left( k_{2}\frac{d(P_{1}^{r}w-w)}{dr}\frac{dw_{\Delta r}}{dr}%
+\lambda (P_{1}^{r}w-w)w_{\Delta r}\right) r^{2}dr=0\ \ \forall
w_{\Delta r}\in V_{\Delta r}^{(1)}[0,R],  \label{pre6}
\end{equation}%
with $\lambda >0$; as before, a good choice now is
$\lambda=\underline{k_{2}}$. By virtue of (\ref{pre2}) it follows
that there exists a constant $C$ independent of $\Delta r$ such that%
\begin{equation}
\left\Vert w-P_{1}^{r}w\right\Vert _{L_{r}^{2}(0,R)}+\Delta r\left\Vert
\frac{d\left( w-P_{1}^{r}w\right) }{dr}\right\Vert _{L_{r}^{2}(0,R)}\leq
C\Delta r^{2}\left\vert w\right\vert _{H_{r}^{2}(0,R)}.  \label{pre7}
\end{equation}%
$P_{1}^{r}$ can be extended to functions of $x$ and $r$ in an
$L^{2}$-sense. Thus, for $(x,r)\in D_{2}\times (0,R(\cdot ))$ we
define the extended projection $P_{1}^{r}:H_{\ \
r}^{1,1}(D_{2}\times (0,R(\cdot )))\rightarrow L^{2}(D_{2})\otimes
V_{\Delta
r}^{(1)}[0,R(\cdot )]$ as%
\begin{equation}
\int_{D_{2}}\int_{0}^{R(x)}\left( k_{2}\frac{\partial (P_{1}^{r}w-w)}{\partial r%
}\frac{dw_{\Delta r}}{dr}+\lambda (P_{1}^{r}w-w)w_{\Delta r}\right)
r^{2}drdx=0\ \ \forall w_{\Delta r}\in L^{2}(D_{2})\otimes V_{\Delta
r}^{(1)}[0,R(\cdot )].  \label{pre8}
\end{equation}%
Assuming that $w\in H_{\ \ r}^{1,1}(D_{2}\times (0,R(\cdot ))$ is such that
for a.e. $x\in D_{2}$, $w(x,\cdot )\in H_{r}^{2}(0,R(x))$, then by virtue of (\ref%
{pre7})%
\begin{equation*}
\left\Vert \left( w-P_{1}^{r}w\right) (x,\cdot )\right\Vert
_{L_{r}^{2}(0,R(x))}+\Delta r\left\Vert \frac{\left( \partial \left(
w-P_{1}^{r}w\right) \right) (x,\cdot )}{\partial r}\right\Vert
_{L_{r}^{2}(0,R(x))}\leq C\Delta r^{2}\left\vert w(x,\cdot )\right\vert
_{H_{r}^{2}(0,R(x))}.
\end{equation*}%
Noting that $H_{\ \ r}^{0,q}(D_{2}\times (0,R(\cdot
)))=L^{2}(D_{2},L_{r}^{2}(0,R(\cdot )))\cap L^{2}(D_{2},H_{r}^{q}(0,R(\cdot )))$%
, then it readily follows that for $q=0,1$%
\begin{equation}
\left\Vert w-P_{1}^{r}w\right\Vert _{H_{\ \ r}^{0,q}(D_{2}\times
(0,R(\cdot )))}\leq C\Delta r^{2-q}\left\Vert w\right\Vert _{H_{\ \
r}^{0,2}(D_{2}\times (0,R(\cdot )))}.  \label{pre10}
\end{equation}
We shall also consider the $x$-Lagrange interpolant for functions
that depend on $x$ and $r$, $I_{0}^{x}:H_{\ \ r}^{1,1}(D_{2}\times
(0,R(\cdot )))\rightarrow V_{h}^{(0)}(\overline{D}_{2})\otimes
H_{r}^{1}(0,R(\cdot ))$. Thus,
for $v(x;r)\in H_{\ \ r}^{1,1}(D_{2}\times (0,R(\cdot )))$%
\begin{equation} \label{pre4.1}
I_{0}^{x}v(x;r)=\sum_{l=1}^{M_{2}}v(x_{l};r)\chi _{l}(x) \
\textrm{with} \ r\in (0,R(x_{l})).
\end{equation}%
Since $I_{0}^{x}$ can be viewed as an extended Lagrange interpolant $%
I_{h}^{(0)}:C(\overline{D}_{2})\rightarrow
V_{h}^{(0)}(\overline{D}_{2})$
, then based on (\ref{pre4}) one can show that%
\begin{equation}
\left\Vert v-I_{0}^{x}v\right\Vert _{L^{2}(D_{2},L_{r}^{2}(0,R(\cdot
)))}\leq Ch\left\Vert v\right\Vert
_{H^{1}(D_{2};L_{r}^{2}(0,R_{s}(\cdot )))}.  \label{uv3}
\end{equation}

\begin{lemma}
\label{lem1} Let $I=(a,b),\ 0\leq a<b$, be a bounded interval with
$\overline{I}=[a,b]$, and let $f\in H^{1}(I)$. There exists an
arbitrarily small number $\epsilon $ and a positive constant
$C(\epsilon )$
such that%
\begin{equation}
\left\Vert f\right\Vert _{L^{\infty }(I)}\leq \epsilon \left\Vert \frac{df}{%
dx}\right\Vert _{L^{2}(I)}+C(\epsilon )\left\Vert f\right\Vert _{L^{2}(I)}.
\label{pre11}
\end{equation}
\end{lemma}

\begin{proof}
Since $H^{1}(I)\hookrightarrow C(\overline{I})$, then $f\in C(
\overline{I})$ and so does $f^{2}$, so for any $y,\ z\in I$, $y<z$,
we have
that%
\begin{equation*}
\begin{array}{r}
f^{2}(z)-f^{2}(y)=\displaystyle\int_{y}^{z}\frac{df^{2}}{dx}dx=2\displaystyle%
\int_{y}^{z}f\frac{df}{dx}dx \\
\\
\leq \epsilon ^{2}\displaystyle\int_{I}\left\vert \frac{df}{dx}\right\vert
^{2}dx+\epsilon ^{-2}\int_{I}f^{2}dx\text{.}%
\end{array}%
\end{equation*}%
Since exists $x^{\ast }\in I$ such that $f^{2}(x^{\ast })=\min_{x\in
I}f^{2}(x)$, then letting $y=x^{\ast }$ it follows that%
\begin{equation*}
f^{2}(y)\leq \frac{1}{\left\vert I\right\vert }\int_{I}f^{2}dx
\end{equation*}%
Substituting this estimate the result follows.
\end{proof}

The next result is a rewording of Lemma 2.4 of \cite{SE}. Let
$F^{1}(R)$ be the closure of $C^{\infty }-$ functions with respect
to the $H_{r}^{1}(0,R)$-norm and with the property that their with
first derivative vanishes at $r=0$.

\begin{lemma}
\label{lem2} If $v\in F^{1}(R),$ then for all $0<a<R,$

1)\ $v\in H^{1}(a,R).$

2)\ There exists an arbitrarily small number $\epsilon $ and a positive
(possibly large) constant $C(\epsilon )$, both depending on $a$, such that%
\begin{equation}
\left\Vert v\right\Vert _{L^{\infty }(a,R)}\leq \epsilon \left\Vert \frac{dv%
}{dr}\right\Vert _{L_{r}^{2}(0,R)}+C(\epsilon )\left\Vert v\right\Vert
_{L_{r}^{2}(0,R)}.  \label{pre12}
\end{equation}
\end{lemma}

\begin{proof}
To prove 1) we note that for any $v\in F^{1}(R)$,
\begin{equation}
\begin{array}{l}
\left\Vert v\right\Vert _{H_{r}^{1}(0,R)}^{2}=\displaystyle%
\int_{0}^{R}v^{2}r^{2}dr+\displaystyle\int_{0}^{R}\left\vert \frac{dv}{dr}%
\right\vert ^{2}r^{2}dr \\
\\
\geq \displaystyle\int_{a}^{R}v^{2}r^{2}dr+\displaystyle\int_{a}^{R}\left%
\vert \frac{dv}{dr}\right\vert ^{2}r^{2}dr \\
\\
\geq a^{2}\displaystyle\int_{a}^{R}v^{2}dr+\displaystyle\int_{a}^{R}\left%
\vert \frac{dv}{dr}\right\vert ^{2}dr=a^{2}\left\Vert v\right\Vert
_{H^{1}(a,R)}^{2}.%
\end{array}
\label{pre13}
\end{equation}%
So, any sequence $\{v_{n}\}$ that converges with respect to the $%
H_{r}^{1}(0,R)-$norm also converges with respect to the $H^{1}(a,R)-$norm.
As for the point 2), we notice that from (\ref{pre11}) and (\ref{pre13}) it
readily follows (\ref{pre12}).
\end{proof}

\begin{lemma} \label{lem3}
For each $(x,t)\in D_{2}\times [0,T_{\mathrm{end}}]$ we
have the following estimates.%
\begin{equation}
\begin{array}{l}
\left\vert i_{0}-i_{0h}\right\vert \leq C\left\vert
u-u_{h}\right\vert
+C\left\vert v_{s}-v_{hs}\right\vert , \\
\\
\left\vert \ln u-\ln u_{h}\right\vert \leq C\left\vert u-u_{h}\right\vert ,
\\
\\
\left\vert \overline{U}(v_{s})-\overline{U}(v_{sh})\right\vert \leq
C\left\vert v_{s}-v_{sh}\right\vert .%
\end{array}
\label{pre10.1}
\end{equation}
\end{lemma}

\begin{proof}
Noting that the functions $x\rightarrow \sqrt{x},\ x\rightarrow
\sqrt{1-x},\ x\rightarrow \ln x$ and $x\rightarrow
\overline{U}(x)$ are smooth bounded and Lipschitz functions in any bounded interval $%
[a,b],\ 0<a<b$, and that the composition and multiplication of
bounded Lipschitz functions results in a Lipschitz function, then
the estimates follow. The constant $C$ in (\ref{pre10.1}) depends on
the constants $P$ and $Q$ of (\ref{estK1}).
\end{proof}

\begin{lemma}
\label{lem4}Let us consider $J$ and its approximate $J_{h}$, then
for a.e. $t\in [0,T_{end}]$ there exists a
positive constant $C$ such that%
\begin{equation}
\begin{array}{r}
\left\Vert J-J_{h}\right\Vert _{L^{2}(D_{2})}^{2}\leq C\left\{ \left\Vert
\phi _{2}(t)-\phi _{2h}(t)\right\Vert _{L^{2}(D_{2})}^{2}+\left\Vert \phi
_{1}(t)-\phi _{1h}(t)\right\Vert _{L^{2}(D_{1})}^{2}\right. \\
\\
\left. +\left\Vert u(t)-u_{h}(t)\right\Vert _{L^{2}(D_{1})}^{2}+\left\Vert
v_{s}(t)-v_{sh}(t)\right\Vert _{L^{2}(D_{2})}^{2}\right\} .%
\end{array}
\label{pre14}
\end{equation}
\end{lemma}

\begin{proof}
Recalling the expressions for $J$, see (\ref{i0})-(\ref{F}), and
$J_{h}$, see (\ref{num6}), using the assumption \textbf{A4} and the
bounds (\ref{estK1}) and (\ref{estK2}),
we have that for all $(x,t)\in D_{2}\times \lbrack 0,T_{%
\mathrm{end}}]$%
\begin{equation*}
\begin{array}{l}
\left\vert J-J_{h}\right\vert =\left\vert a_{2}i_{0}\sinh (\beta \eta
)-a_{2}i_{0h}\sinh (\beta \eta _{h})\right\vert \\
\\
\leq \left\vert a_{2}i_{0}\left( \sinh (\beta \eta )-\sinh (\beta \eta
_{h})\right) \right\vert +\left\vert a_{2}\left( i_{0}-i_{0h}\right) \sinh
(\beta \eta _{h})\right\vert \\
\\
\leq C\left\vert \sinh (\beta \eta )-\sinh (\beta \eta
_{h})\right\vert +\left\Vert a_{2}\sinh (\beta \eta _{h})\right\Vert
_{L^{\infty }(D_{2}\times [0,T_{end}])}\left\vert
i_{0}-i_{0h}\right\vert \\
\\
\leq C\left\vert \sinh (\beta \eta )-\sinh (\beta \eta
_{h})\right\vert +C\left\vert i_{0}-i_{0h}\right\vert.
\end{array}%
\end{equation*}
where due to the bounds (\ref{estK1}) and (\ref{estK2}) the
constants $C=C(P,Q,K)$. Now, by virtue of the mean value theorem
there exists $z\in (\eta ,\eta _{h})$
such that%
\begin{equation*}
\left\vert \sinh (\beta \eta )-\sinh (\beta \eta _{h})\right\vert
\leq \beta \left\vert \cosh (z)\right\vert \left\vert \eta -\eta
_{h}\right\vert \leq C\left\vert \eta -\eta _{h}\right\vert,
\end{equation*}%
and resorting again to (\ref{estK1}) and (\ref{estK2}) it follows
that
\begin{equation*}
\left\Vert z\right\Vert _{L^{\infty }(D_{2}\times \lbrack 0,T_{\mathrm{end}%
}])}\leq \left\Vert \eta \right\Vert _{L^{\infty }(D_{2}\times \lbrack 0,T_{%
\mathrm{end}}])}+\left\Vert \eta _{h}\right\Vert _{L^{\infty }(D_{2}\times
\lbrack 0,T_{\mathrm{end}}])}\leq C
\end{equation*}%
Hence, applying Lemma \ref{lem3} yields
\begin{equation*}
\left\vert J-J_{h}\right\vert \leq C\left( \left\vert \phi _{2}-\phi
_{2h}\right\vert +\left\vert \phi _{1}-\phi _{1h}\right\vert +\left\vert
u-u_{h}\right\vert +\left\vert v_{s}-v_{sh}\right\vert \right)
\end{equation*}%
From this estimate it follows (\ref{pre14}).
\end{proof}

\subsection{Error estimates for the potentials}

To estimate the error for the potentials $\phi _{1}$ and $\phi _{2}$
is convenient to introduce the spaces $V:=H^{1}(D_{1})\times
H^{1}(D_{2}):=\left\{ w=\left( w_{1},w_{2}\right) :w_{1}\in
H^{1}(D_{1}),\
w_{2}\in H^{1}(D_{2})\right\} $ and $V_{h}\subset V$, where $%
V_{h}:=V_{h}^{(1)}(\overline{D}_{1})\times V_{h}^{(1)}(\overline{D}_{2})$. $%
V $ is a Hilbert space with norm
\begin{equation*}
\left\Vert w\right\Vert _{V}=\left( \left\Vert w_{1}\right\Vert
_{H^{1}(D_{1})}^{2}+\left\Vert w_{2}\right\Vert _{H^{1}(D_{2})}^{2}\right)
^{1/2},
\end{equation*}%
and seminorm%
\begin{equation*}
\left\vert v\right\vert _{V}=\left( \left\vert w_{1}\right\vert
_{H^{1}(D_{1})}^{2}+\left\vert w_{2}\right\vert _{H^{1}(D_{2})}^{2}\right)
^{1/2}.
\end{equation*}%
Considering the bilinear forms $a_{1}:H^{1}(D_{1})\times
H^{1}(D_{1})\rightarrow \mathbb{R}$ and $a_{2}:H^{1}(D_{2})\times
H^{1}(D_{2})\rightarrow \mathbb{R}$,
\begin{equation*}
\left\{
\begin{array}{l}
a_{1}(\phi _{1},\psi _{1})=\displaystyle\int_{D_{1}}\kappa
(u)\frac{d\phi _{1}}{dx}\frac{d\psi _{1}}{dx}dx,\ \  \\
\\
a_{2}(\phi _{2},\psi _{2})=\displaystyle\int_{D_{2}}\sigma \frac{d\phi _{2}}{dx}\frac{d\psi _{2}}{dx}dx,%
\end{array}%
\right.
\end{equation*}%
we can define the bilinear form $a:V\times V\rightarrow \mathbb{R}$
as follows. Let $\Phi $ and $\Psi $ $\in V$, $\Phi =(\phi _{1},\phi
_{2})$ and
$\Psi =(\psi _{1},\psi _{2})$, then%
\begin{equation*}
a(\Phi ,\Psi )=a_{1}(\phi _{1},\psi _{1})+a_{2}(\phi _{2},\psi _{2}).
\end{equation*}%
Furthermore, concerning the right hand side terms of (\ref{w3}) and (\ref{w4}%
), we introduce the operator $B:V\rightarrow V^{\ast }$, $V^{\ast }$ being
the dual for $V$, as%
\begin{equation*}
\left\langle B(\Phi ),\Psi \right\rangle =\int_{D_{2}}J\left( \psi _{2}-\psi
_{1}\right) dx.
\end{equation*}%
Hence, we can recast the equations (\ref{w3}) and (\ref{w4}) as follows.
Find $\Phi \in L^{2}(0,T_{end}; W(D_{1})\times H^{1}(D_{2}))$ such that%
\begin{equation}
a(\Phi ,\Psi )+\left\langle B(\Phi ),\Psi \right\rangle =-\int_{D_{2}}g\psi
_{2}dx\ \ \forall \Psi \in V.  \label{pre15}
\end{equation}%
Likewise, the finite element solutions $\phi _{1h}$ and $\phi _{2h}$
that satisfy (\ref{num3}) and (\ref{num4}) respectively, can be
formulated as follows. For all $t\in [0,T_{end}]$, find $\Phi _{h}\in W_{h}(\overline{D}_{1})\times V_{h}^{(1)}(%
\overline{D}_{2})$ such that%
\begin{equation}
a_{h}(\Phi _{h},\Psi _{h})+\left\langle B_{h}(\Phi _{h}),\Psi
_{h}\right\rangle =-\int_{D_{2}}g\psi _{2h}dx\ \ \forall \Psi _{h}\in V_{h},
\label{eep1}
\end{equation}%
where%
\begin{equation}
a_{h}(\Phi _{h},\Psi _{h})=a_{1h}(\phi _{1h},\psi _{1h})+a_{2}(\phi
_{2h},\psi _{2h}),  \label{eep1_1}
\end{equation}%
with%
\begin{equation*}
a_{1h}(\phi _{1h},\psi _{1h})=\int_{D_{1}}\kappa (u_{h})\frac{d\phi
_{1h}}{dx}\frac{d\psi _{1h}}{dx}dx,
\end{equation*}%
and%
\begin{equation}
\left\langle B_{h}(\Phi _{h}),\Psi _{h}\right\rangle
=\int_{D_{2}}J_{h}\left( \psi _{2h}-\psi _{1h}\right) dx.  \label{eep1_2}
\end{equation}

\begin{remark} \label{solpots}

The existence and uniqueness of $\Phi$ is proven in \cite{Diaz}
under the same kind of assumptions as \textbf{A1}-\textbf{A4},
whereas in \cite{WXZ} and \cite{Kro} the existence is proven
applying Schauder Fixed Point Theorem \cite{GT} and the uniqueness
using the fact that $\phi_{1}(x) \in W(D_{1})$.

\end{remark}
As for the bilinear form $a$ and the operator $B$, we have the
following result.

\begin{lemma}
\label{lem5} Assuming that \textbf{A1-A4} hold, we have that: (i)\
the
bilinear form $a$ is continuous, (ii)\ the operator $B$ is monotone, i.e.,%
\begin{equation}
\left\langle B(\Phi )-B(\widehat{\Phi }),\Phi -\widehat{\Phi
}\right\rangle \geq 0,  \label{pre16}
\end{equation}%
bounded and continuous in the sense that for all $\Psi \in V$ there exists a
constant $C$ such that%
\begin{equation}
\left\langle B(\Phi )-B(\widehat{\Phi }),\Psi \right\rangle \leq
C\left( \left\Vert \Phi -\widehat{\Phi }\right\Vert _{V}\right)
\left( \left\Vert \psi _{1}\right\Vert _{L^{2}(D_{1})}+\left\Vert
\psi _{2}\right\Vert _{L^{2}(D_{2})}\right) .  \label{pre17}
\end{equation}
\end{lemma}

\begin{proof}
It is easy to prove the continuity of the bilinear form $a$ if one takes
into account the regularity assumption\ \textbf{A2}. To prove (\ref{pre16})
we note that%
\begin{equation*}
\left\langle B(\Phi )-B(\widehat{\Phi }),\Phi -\widehat{\Phi }\right\rangle
=\int_{D_{2}}a_{2}i_{0}\left( \sinh (\beta \eta )-\sinh (\beta \widehat{\eta
})\right) \left( \left( \phi _{2}-\widehat{\phi }_{2}\right) -\left( \phi
_{1}-\widehat{\phi }_{1}\right) \right) dx,
\end{equation*}%
where $\eta =\phi _{2}-\phi _{1}-\alpha \ln u-\overline{U}$ and $\widehat{%
\eta }=\widehat{\phi }_{2}-\widehat{\phi }_{1}-\alpha \ln
u-\overline{U}$. Since for all $x\in D_{2}$ $a_{2}i_{0}>0$, then by
virtue of \textbf{A4} we can choose a constant $C(P,Q)$ such that
for all $x\in D_{2}$ $C(P,Q)\leq a_{2}i_{0}$, and by the mean value
theorem $\sinh (\beta \eta )-\sinh (\beta \widehat{\eta })\geq \beta
(\eta -\widehat{\eta })$, then one readily
obtains%
\begin{equation}
\left\langle B(\Phi )-B(\widehat{\Phi }),\Phi -\widehat{\Phi }\right\rangle
\geq C\int_{D_{2}}\left( \left( \phi _{2}-\phi _{1}\right) -\left( \widehat{%
\phi }_{2}-\widehat{\phi }_{1}\right) \right) ^{2}dx\geq 0.
\label{pre17.1_1}
\end{equation}%
To prove that $B$ is bounded we notice that for all $\Phi \in V$%
\begin{equation*}
\left\langle B(\Phi ),\Phi \right\rangle \leq \int_{D_{2}}\left\vert
J\right\vert \times \left\vert \phi _{2}-\phi _{1}\right\vert
dx=\int_{D_{2}}\left\vert a_{2}i_{0}(\sinh (\beta \eta )\right\vert \times
\left\vert \phi _{2}-\phi _{1}\right\vert dx,
\end{equation*}%
but $\left\vert J\right\vert $ is bounded by virtue of \textbf{A4},
then using the Cauchy-Schwarz inequality it readily follows that
there exists a
bounded positive constant $C$ such that%
\begin{equation*}
\left\langle B(\Phi ),\Phi \right\rangle \leq C\left\Vert \Phi \right\Vert
_{V},
\end{equation*}%
so $B$ is bounded. To prove that $B$ is continuous, we again notice that
\begin{equation*}
\left\langle B(\Phi )-B(\widehat{\Phi }),\Psi \right\rangle \leq
\int_{D_{2}}\left\vert a_{2}i_{0}(\sinh (\beta \eta )-\sinh (\beta \widehat{%
\eta }))\right\vert \times \left\vert \psi _{2}-\psi _{1}\right\vert dx,
\end{equation*}%
so, arguing as in the proof of Lemma \ref{lem4} we have that there exists a
positive constant $C$ such that%
\begin{equation*}
\left\vert a_{2}i_{0}(\sinh (\beta \eta )-\sinh (\beta \widehat{\eta }%
))\right\vert \leq C\left\vert \eta -\widehat{\eta }\right\vert =C\left\vert
\left( \phi _{2}-\widehat{\phi }_{2}\right) -\left( \phi _{1}-\widehat{\phi }%
_{1}\right) \right\vert .
\end{equation*}%
Substituting this estimate in the above inequality and making use of the
Cauchy-Schwarz inequality it follows that%
\begin{equation*}
\begin{array}{c}
\left\langle B(\Phi )-B(\widehat{\Phi }),\Psi \right\rangle \leq C\left(
\left\Vert \phi _{2}-\widehat{\phi }_{2}\right\Vert
_{L^{2}(D_{2})}+\left\Vert \phi _{1}-\widehat{\phi }_{1}\right\Vert
_{L^{2}(D_{1})}\right) \left( \left\Vert \psi _{2}\right\Vert
_{L^{2}(D_{2})}+\left\Vert \psi _{1}\right\Vert _{L^{2}(D1)}\right) \\
\\
\leq C\left\Vert \Phi -\widehat{\Phi }\right\Vert _{V}\left( \left\Vert \psi
_{2}\right\Vert _{L^{2}(D_{2})}+\left\Vert \psi _{1}\right\Vert
_{L^{2}(D1)}\right) .%
\end{array}%
\end{equation*}
\end{proof}

\begin{corollary}
\label{cor1} The discrete bilinear form $a_{h}$ defined in (\ref{eep1_1}) is
continuous. The discrete operator $B_{h}$ defined in (\ref{eep1_2}) is
monotone, bounded and continuous.
\end{corollary}

\begin{theorem}
\label{Theorem 1} For a.e. $t\in \lbrack 0,T_{\mathrm{end}}]$, let
the solution of (\ref{pre15}), $\Phi (t)=\left( \phi _{1}(t),\phi
_{2}(t)\right)$, be in $H^{2}(D_{1})\times H^{2}(D_{2})$. There
exists a constant $C$
independent of $h$ such that%
\begin{equation}
\begin{array}{c}
\left\Vert \Phi (t)-\Phi _{h}(t)\right\Vert _{V}^{2}\leq C\left(
h^{2}(\left\Vert \phi _{1}(t)\right\Vert _{H^{2}(D_{1})}^{2}+\left\Vert \phi
_{2}(t)\right\Vert _{H^{2}(D_{2})}^{2})\right. \\
\\
\left. +\left\Vert u(t)-u_{h}(t)\right\Vert _{L^{2}(D_{1})}^{2}+\left\Vert
v_{s}(t)-v_{sh}(t)\right\Vert _{L^{2}(D_{2})}^{2}\right) .%
\end{array}
\label{eep}
\end{equation}
\end{theorem}

\begin{proof}
Setting $\Psi =\Psi _{h}$ in (\ref{pre15}) and subtracting (\ref{eep1})
yields%
\begin{equation*}
a(\Phi ,\Psi _{h})-a_{h}(\Phi _{h},\Psi _{h})+\left\langle B(\Phi )-B\left(
\Phi _{h}\right) ,\Psi _{h}\right\rangle =-\left\langle B(\Phi
_{h})-B_{h}\left( \Phi _{h}\right) ,\Psi _{h}\right\rangle .
\end{equation*}%
Noting that%
\begin{equation*}
a(\Phi ,\Psi _{h})-a_{h}(\Phi _{h},\Psi _{h})=a_{h}(\Phi -\Phi
_{h},\Psi _{h})+\int_{D_{1}}\left( \kappa (u)-\kappa (u_{h})\right)
\frac{\partial \phi _{1}}{\partial x}\frac{d\psi _{1h}}{dx}dx,
\end{equation*}%
it follows that%
\begin{equation}
\begin{array}{r}
a_{h}(\Phi -\Phi _{h},\Psi _{h})+\left\langle B(\Phi )-B\left( \Phi
_{h}\right) ,\Psi _{h}\right\rangle =\left\langle B_{h}(\Phi
_{h})-B\left( \Phi
_{h}\right) ,\Psi _{h}\right\rangle \\
\\
-\displaystyle\int_{D_{1}}\left( \kappa (u)-\kappa (u_{h})\right) \frac{\partial \phi _{1}}{\partial x}\frac{d\psi _{1h}}{dx}dx.%
\end{array}
\label{eep2}
\end{equation}%
To estimate the terms of this expression we choose $\Psi
_{h}=I_{h}^{(1)}\Phi -\Phi _{h}=$ $\left( I_{h}^{(1)}\phi _{1}-\phi
_{1h},I_{h}^{(1)}\phi _{2}-\phi _{2h}\right) $, $I_{h}^{(1)}$ being
the Lagrange interpolant on
$V_{h}=V_{h}^{(1)}(\overline{D}_{1})\times
V_{h}^{(1)}(\overline{D}_{2})$, this means that $ I_{h}^{(1)}\phi
_{1}\in V_{h}^{(1)}(\overline{D}_{1})$ and $I_{h}^{(1)}\phi _{2}\in
V_{h}^{(1)}(\overline{D}_{2})$. For convenience, we shall split the
expression for $\Psi _{h}$ as
\begin{equation*}
\Psi _{h}=\left( \Phi -\Phi _{h}\right) +\left( I_{h}^{(1)}\Phi
-\Phi \right) .
\end{equation*}
Replacing this expression for $\Psi _{h}$ in (\ref{eep2}) we have that
\begin{equation}
\left\{
\begin{array}{l}
a_{h}(\Phi -\Phi _{h},\Phi -\Phi _{h})+\left\langle B(\Phi )-B\left(
\Phi
_{h}\right) ,\Phi -\Phi _{h}\right\rangle \\
\\
=a_{h}(\Phi -\Phi _{h},\Phi -I_{h}^{(1)}\Phi )+\left\langle B(\Phi
)-B\left( \Phi
_{h}\right) ,\Phi -I_{h}^{(1)}\Phi \right\rangle \\
\\
+\left\langle B_{h}(\Phi _{h})-B\left( \Phi _{h}\right) ,\left( \Phi -\Phi
_{h}\right) +\left( I_{h}^{(1)}\Phi -\Phi \right) \right\rangle \\
\\
+\displaystyle\int_{D_{1}}\left( \kappa (u)-\kappa (u_{h})\right)
\frac{\partial \phi _{1}}{\partial x}\frac{\partial \left( \left(
\phi _{1}-\phi
_{1h}\right) +\left( I_{h}^{(1)}\phi _{1}-\phi _{1h}\right) \right) }{\partial x}%
dx.%
\end{array}%
\right.  \label{eep3}
\end{equation}%
We bound the terms of (\ref{eep3}). We start by showing that there
exists a positive constant $\alpha$ such that the term on the left
hand side satisfies
\begin{equation}
a_{h}(\Phi -\Phi _{h},\Phi -\Phi _{h})+\left\langle B(\Phi )-B\left(
\Phi _{h}\right) ,\Phi -\Phi _{h}\right\rangle \geq \alpha
\left\Vert \Phi -\Phi _{h}\right\Vert _{V}^{2}. \label{eep4}
\end{equation}
To do so we note that by virtue of (\ref{pre17.1_1})%
\begin{equation}
\begin{array}{l}
a_{h}(\Phi -\Phi _{h},\Phi -\Phi _{h})+\left\langle B(\Phi )-B\left(
\Phi
_{h}\right) ,\Phi -\Phi _{h}\right\rangle \\
\\
\geq a_{1h}(\phi _{1}-\phi _{1h},\phi _{1}-\phi _{1h})+a_{2}(\phi
_{2}-\phi
_{2h},\phi _{2}-\phi _{2h}) \\
\\
+C\displaystyle\int_{D_{2}}\left( \left( \phi _{2}-\phi _{1}\right)
-\left( \phi
_{2h}-\phi _{1h}\right) \right) ^{2}dx.%
\end{array}
\label{eep3.1}
\end{equation}%
Since $\phi _{1}-\phi _{1h}\in W(D_{1})$, we can use \textbf{A2 }and
Poincar\'{e}-Wirtinger inequality to bound $a_{1}(\phi _{1}-\phi
_{1h},\phi _{1}-\phi _{1h})$ from below as
\begin{equation*}
a_{1h}(\phi _{1}-\phi _{1h},\phi _{1}-\phi _{1h})\geq
c_{1}\left\Vert \phi _{1}-\phi _{1h}\right\Vert _{H^{1}(D_{1})}^{2},
\end{equation*}%
where the constant $c_{1}=\kappa _{0}\left( 1+C_{P}\right) ^{-1}$,
$C_{P}$
being the constant of the Poincar\'{e}-Wirtinger inequality; using again \textbf{%
A2}, we bound\textbf{\ }the term $a_{2}(\phi _{2}-\phi _{2h},\phi _{2}-\phi
_{2h})$ as
\begin{equation*}
a_{2}(\phi _{2}-\phi _{2h},\phi _{2}-\phi _{2h})\geq \sigma _{0}\left\vert
\phi _{2}-\phi _{2h}\right\vert _{H^{1}(D_{2})}^{2}.
\end{equation*}%
Applying Young inequality we find that there exists a constant
$\gamma \in \left( 0,1\right) $ such that
\begin{equation*}
\begin{array}{l}
C\displaystyle\int_{D_{2}}\left( \left( \phi _{2}-\phi _{1}\right)
-\left( \phi
_{2h}-\phi _{1h}\right) \right) ^{2}dx. \\
\\
\geq C\left[ (1-\frac{4}{\gamma })\left\Vert \phi _{1}-\phi
_{1h}\right\Vert _{L^{2}(D_{2})}^{2}+(1-\gamma )\left\Vert \phi
_{2}-\phi _{2h}\right\Vert
_{L^{2}(D_{2})}^{2}\right] \\
\\
\geq C\left[ (1-\frac{4}{\gamma })\left\Vert \phi _{1}-\phi
_{1h}\right\Vert _{L^{2}(D_{1})}^{2}+(1-\gamma )\left\Vert \phi
_{2}-\phi _{2h}\right\Vert
_{L^{2}(D_{2})}^{2}\right].%
\end{array}%
\end{equation*}%
Now, we can choose the constants $C$ and $\gamma $ such that $%
c_{2}=c_{1}+C(1-\frac{4}{\gamma })>0$, and substitute these bounds in (\ref%
{eep3.1}) to obtain the inequality (\ref{eep4}), where $\alpha =\min
(c_{1},c_{2},\sigma _{0},(1-\gamma )C)$. Next, we bound the terms on
the right hand side. By continuity of the bilinear form and
Young inequality, we find that there exists a small positive number $%
\epsilon _{1}$ and a constant $C(\epsilon _{1})$ such that%
\begin{equation}
\begin{array}{r}
a_{h}(\Phi -\Phi _{h},\Phi -I_{h}^{(1)}\Phi )\leq C\left\Vert \Phi
-\Phi
_{h}\right\Vert _{V}\left\Vert \Phi -I_{h}^{(1)}\Phi \right\Vert _{V} \\
\\
\leq \epsilon _{1}\left\Vert \Phi -\Phi _{h}\right\Vert _{V}^{2}+C(\epsilon
_{1})\left\Vert \Phi -I_{h}^{(1)}\Phi \right\Vert _{V}^{2}.%
\end{array}
\label{eep5}
\end{equation}%
To bound $\left\langle B(\Phi )-B\left( \Phi _{h}\right) ,\Phi
-I_{h}^{(1)}\Phi \right\rangle $ we note that
\begin{equation*}
\left\langle B(\Phi )-B\left( \Phi _{h}\right) ,\Phi
-I_{h}^{(1)}\Phi \right\rangle \leq \int_{D_{2}}\left\vert
a_{2}i_{0}\right\vert \left\vert \int_{\beta \eta _{h}}^{\beta \eta
}\cosh \xi d\xi \right\vert \left\vert \Phi -I_{h}^{(1)}\Phi
\right\vert dx.
\end{equation*}%
By virtue of (\ref{estK1}) and (\ref{estK2}) and the mean value theorem for
the integral
\begin{equation*}
\begin{array}{l}
\left\langle B(\Phi )-B\left( \Phi _{h}\right) ,\Phi
-I_{h}^{(1)}\Phi \right\rangle \leq
C\displaystyle\int_{D_{2}}\left\vert \eta -\eta
_{h}\right\vert \left\vert \Phi -I_{h}^{(1)}\Phi \right\vert dx \\
\\
\leq C\displaystyle\int_{D_{2}}\left( \left\vert \phi _{2}-\phi
_{2h}\right\vert +\left\vert \phi _{1}-\phi _{1h}\right\vert +\left\vert
u-u_{h}\right\vert +\left\vert v_{s}-v_{sh}\right\vert \right) \left\vert
\Phi -I_{h}^{(1)}\Phi \right\vert dx.%
\end{array}%
\end{equation*}%
Applying Young inequality yields%
\begin{equation}
\begin{array}{c}
\left\langle B(\Phi )-B\left( \Phi _{h}\right) ,\Phi
-I_{h}^{(1)}\Phi \right\rangle \leq \epsilon _{2}\left( \left\Vert
\Phi -\Phi _{h}\right\Vert
_{V}^{2}+\left\Vert u-u_{h}\right\Vert _{L^{2}(D_{1})}^{2}\right. \\
\\
\left. +\left\Vert v_{s}-v_{sh}\right\Vert _{L^{2}(D_{2})}^{2}\right)
+C(\epsilon _{2})\left\Vert \Phi -I_{h}^{(1)}\Phi \right\Vert _{V}^{2},%
\end{array}
\label{eep6}
\end{equation}%
where $\epsilon _{2}$ is a small positive number and $C(\epsilon _{2})$ is a
constant. Next, noting that by virtue of (\ref{estK1}) and (\ref{estK2}) $%
\left\vert \sinh (\beta \eta _{h})\right\vert $ is bounded in $D_{2}$, then
\begin{equation*}
\left\langle B_{h}(\Phi _{h})-B\left( \Phi _{h}\right) ,\left( \Phi
-\Phi _{h}\right) +\left( I_{h}^{(1)}\Phi -\Phi \right)
\right\rangle \leq C\int_{D_{2}}a_{2}\left\vert
i_{0}-i_{0h}\right\vert \left\vert \left( \Phi -\Phi _{h}\right)
+\left( I_{h}^{(1)}\Phi -\Phi \right) \right\vert dx.
\end{equation*}%
Again, using Lemma \ref{lem3} and Young inequality we obtain that there
exist a small number $\epsilon _{3}$ and a constant $C(\epsilon _{3})$ such
that%
\begin{equation}
\begin{array}{l}
\left\langle B_{h}(\Phi _{h})-B\left( \Phi _{h}\right) ,\left( \Phi
-\Phi _{h}\right) +\left( I_{h}^{(1)}\Phi -\Phi \right)
\right\rangle \leq \epsilon
_{3}\left\Vert \Phi -\Phi _{h}\right\Vert _{V}^{2} \\
\\
+\epsilon _{3}\left( \left\Vert u-u_{h}\right\Vert
_{L^{2}(D_{1})}^{2}+\left\Vert v_{s}-v_{sh}\right\Vert
_{L^{2}(D_{2})}^{2}\right) +C(\epsilon _{3})\left\Vert \Phi
-I_{h}^{(1)}\Phi
\right\Vert _{V}^{2}.%
\end{array}
\label{eep7}
\end{equation}%
To bound the last term on the right hand side of (\ref{eep3}) we note that $%
\left\Vert \frac{\partial \phi _{1}}{\partial x}\right\Vert
_{L^{\infty }(D_{1})}$ is bounded and by virtue of assumption
\textbf{A4}, and  by the mean value theorem, $\left\vert \left(
\kappa (u)-\kappa (u_{h})\right) \right\vert \leq C\left\vert
u-u_{h}\right\vert $, then\textbf{\ }it follows
that%
\begin{equation}
\begin{array}{l}
\displaystyle\int_{D_{1}}\left( \kappa (u)-\kappa (u_{h})\right)
\frac{\partial \phi _{1}}{\partial x}\frac{\partial \left( \left(
\phi _{1}-\phi _{1h}\right) +\left( I_{h}^{(1)}\phi _{1}-\phi
_{1h}\right) \right)
}{\partial x}dx \\
\\
\leq C\displaystyle\int_{D_{1}}\left\vert u-u_{h}\right\vert
\left\vert \frac{\partial \left( \left( \phi _{1}-\phi _{1h}\right)
+\left( I_{h}^{(1)}\phi
_{1}-\phi _{1h}\right) \right) }{\partial x}\right\vert dx \\
\\
\leq \epsilon _{4}\left\Vert \Phi -\Phi _{h}\right\Vert _{V}^{2}+C(\epsilon
_{4})\left( \left\Vert u-u_{h}\right\Vert _{L^{2}(D_{1})}^{2}+\left\Vert
\Phi -I_{h}^{(1)}\Phi \right\Vert _{V}^{2}\right) .%
\end{array}
\label{eep8}
\end{equation}%
Letting$\ \epsilon _{1}+\cdots +\epsilon _{4}=\alpha /2$ and noting that,
see (\ref{pre4.1}),%
\begin{equation*}
\begin{array}{c}
\left\Vert \Phi -I_{h}^{(1)}\Phi \right\Vert _{V}^{2}=\left\Vert
\phi _{1}-I_{h}^{(1)}\phi _{1}\right\Vert
_{H^{1}(D_{1})}^{2}+\left\Vert \phi
_{2}-I_{h}^{(1)}\phi _{2}\right\Vert _{H^{1}(D_{2})}^{2} \\
\\
\leq Ch^{2}\left( \left\Vert \phi _{1}\right\Vert
_{H^{2}(D_{1})}^{2}+\left\Vert \phi _{2}\right\Vert
_{H^{2}(D_{2})}^{2}\right) .%
\end{array}%
\end{equation*}%
the estimate (\ref{eep}) follows from (\ref{eep3})-(\ref{eep8}).
\end{proof}

\subsection{Error estimates for the concentrations $u(x,t)$ and $v(x;r,t)$}

{We wish to estimate $u(x,t)-u_{h}(x,t)$ and $v(x;r,t)-v_{h\Delta
r}(x;r,t)$ in the $L^{2}-$norm assuming that both $u(x,t)$ and
$v(x;r,t)$ are as
regular as required. Following the standard approach, we decompose $%
u-u_{h}$ as%
\begin{equation} \label{uv1}
u-u_{h}=(u-P_{1}u)+(P_{1}u-u_{h})\equiv \rho _{u}+\theta _{u},
\end{equation}%
where $P_{1}$ is the elliptic projector defined in (\ref{pre3}), and
note that $\theta _{u}(x,t)\in V_{h}^{(1)}(\overline{D}_{1})$. To
carry out a decomposition of this kind for $v(x;r,t)-v_{h\Delta
r}(x;r,t)$, at first we can try using the extended elliptic
projector $P_{1}^{r}:H_{\ \ r}^{1,1}(D_{2}\times (0,R_{s}(\cdot
)))\rightarrow L^{2}(D_{2})\otimes V_{\Delta r}^{(1)}[0,R_{s}(\cdot
)]$ defined in (\ref{pre8}) and assume that for all $t$,
$v(x;r,t)\in H_{\ \ r}^{1,1}(D_{2}\times (0,R_{s}(\cdot )))$, then
we find that for a. e. $x\in D_{2}$%
\begin{equation*}
P_{1}^{r}v(x;r,t)=\sum_{j=1}^{M}P_{1}^{r}v(x;r_{j},t)\alpha _{j}(r),
\end{equation*}%
here, $M$ denotes the number of mesh points in $[0,R_{s}(\cdot )]$,
$\left\{ \alpha _{j}(r)\right\} _{j=1}^{M}$ the nodal basis of the
linear finite element space $V_{\Delta r}^{(1)}[0,R_{s}(\cdot
)]\subset H_{r}^{1}(0,R_{s}(\cdot ))$ and the function
$P_{1}^{r}v(x;r_{j},t)\in
L^{2}(D_{2})$; so, in general, $P_{1}^{r}v(x;r,t)$ is not in $V_{h\Delta r}(\overline{D}%
_{3})$ and, consequently, it does not make sense to use $P_{1}^{r}v(x;r,t)-v_{h%
\Delta r}(x;r,t)$ for such type of decomposition; however, recalling
the interpolant $I_{0}^{x}:H_{\ \ r}^{1,1}(D_{2}\times (0,R(\cdot
)))\rightarrow V_{h}^{(0)}(\overline{D}_{2})\otimes
H_{r}^{1}(0,R(\cdot ))$ defined in (\ref{pre4.1}) and further
assuming that $P_{1}^{r}v(x;r_{j},t)\in H^{1}(D_{2})$, then it
follows that $P_{1}^{r}v(x;r,t)\in H_{\ \
r}^{1,1}(D_{2}\times (0,R(\cdot )))$ and, therefore, we can define $%
I_{0}^{x}P_{1}^{r}v(x;r,t)$ as
\begin{equation*}
I_{0}^{x}P_{1}^{r}v(x;r,t)=\sum_{l=1}^{M_{2}}%
\sum_{j=1}^{M^{(l)}}P_{1}^{r}v(x_{l};r_{j},t)\alpha
_{j}^{(l)}(r)\chi _{l}(x);
\end{equation*}%
this expression implies that $I_{0}^{x}P_{1}^{r} v(x;r,t)\in
V_{h\Delta r}(\overline{D}_{3})$, so it makes sense to set
\begin{equation*}
\theta _{v}(x;r,t)=I_{0}^{x}P_{1}^{r}v(x;r,t)-v_{h\Delta r}(x;r,t).
\end{equation*}%
Now, using again the extended $P_{1}^{r}$ elliptic projector we
define
\begin{equation*}
\rho _{v}(x;r,t)=v(x;r,t)-P_{1}^{r}v(x;r,t)\in H_{\ \
r}^{1,1}(D_{2}\times (0,R(\cdot ))),
\end{equation*}%
and consequently%
\begin{equation*}
I_{0}^{x}\rho
_{v}(x;r,t)=I_{0}^{x}v(x;r,t)-I_{0}^{x}P_{1}^{r}v(x;r,t).
\end{equation*}%
Then, from all these considerations we can write that
\begin{equation} \label{uv2}
v(x;r,t)-v_{h\Delta r}(x;r,t)=(v-I_{0}^{x}v)(x;r,t)+I_{0}^{x}\rho
_{v}(x;r,t)+\theta _{v}(x;r,t).
\end{equation}
From (\ref{uv1}) and (\ref{uv2}) it follows that
\begin{equation*}
\left\Vert u(t)-u_{h}(t)\right\Vert _{L^{2}(D_{1})}\leq \left\Vert \rho
_{u}(t)\right\Vert _{L^{2}(D_{1})}+\left\Vert \theta _{u}(t)\right\Vert
_{L^{2}(D_{1})},
\end{equation*}%
and%
\begin{equation*}
\begin{array}{r}
\left\Vert v(t)-v_{h\Delta r}(t)\right\Vert
_{L^{2}(D_{2},L_{r}^{2}(0,R(\cdot )))}\leq \left\Vert
v(t)-I_{0}^{x}v(t)\right\Vert _{L^{2}(D_{2},L_{r}^{2}(0,R(\cdot )))} \\
\\
+\left\Vert I_{0}^{x}\rho _{v}(t)\right\Vert
_{L^{2}(D_{2},L_{r}^{2}(0,R(\cdot )))}+\left\Vert \theta _{v}(t)\right\Vert
_{L^{2}(D_{2},L_{r}^{2}(0,R(\cdot )))}.%
\end{array}%
\end{equation*}%
The estimates for $\rho _{u}$ and $\rho _{v}$ are given in (\ref{pre5}) and (%
\ref{pre10}) respectively, i.e.,%
\begin{equation} \label{uv2_1}
\left\{
\begin{array}{l}
\left\Vert \rho _{u}(t)\right\Vert _{L^{2}(D_{1})}\leq Ch^{2}\left\Vert
u(t)\right\Vert _{H^{2}(D_{1})}, \\
\\
\left\Vert \rho _{v}(t)\right\Vert _{L^{2}(D_{2},L_{r}^{2}(0,R(\cdot
)))}\leq C\Delta r^{2}\left\Vert v(t)\right\Vert
_{L^{2}(D_{2},H_{r}^{2}(0,R(\cdot )))},%
\end{array}
\right.
\end{equation}%
then, it remains to calculate the estimates for $\theta _{u}$ and
$\theta _{v} $; but before going into the details of such
calculations, we present
new estimates for $\left\Vert J-J_{h}\right\Vert _{L^{2}(D_{2})}^{2}$ and $%
\left\Vert v_{s}(t)-v_{sh}(t)\right\Vert _{L^{2}(D_{2})}^{2}$, which
depend on $\theta _{u}$ and $\theta _{v}$ respectively, and will be
useful for the subsequent part of the analysis.

\begin{lemma}
\label{lem6} Assuming that the regularity assumptions required in the
estimates hold, there exist an arbitrarily small positive number $\epsilon $
and constants $C$ and $C(\epsilon )$ independent of $h$ and $\Delta r$ such
that%
\begin{equation}
\begin{array}{r}
\left\Vert v_{s}(t)-v_{sh}(t)\right\Vert _{L^{2}(D_{2})}^{2}\leq C
h^{2}\left\vert v_{s}(t)\right\vert
_{H^{1}(D_{2})}^{2}+C(\epsilon)\Delta r^{2}\left\Vert
v(t)\right\Vert _{L^{2}(D_{2},H_{r}^{2}(0,R_{s}(\cdot
)))}^{2} \\
\\
+C(\epsilon )\left\Vert \theta _{v}(t)\right\Vert
_{L^{2}(D_{2},L_{r}^{2}(0,R_{s}(\cdot )))}^{2}+\epsilon \displaystyle%
\left\Vert \frac{\partial \theta _{v}(t)}{\partial r}\right\Vert
_{L^{2}(D_{2},L_{r}^{2}(0,R_{s}(\cdot )))}^{2},%
\end{array}
\label{uv6.0}
\end{equation}%
and

\begin{equation}
\begin{array}{c}
\left\Vert J-J_{h}\right\Vert _{L^{2}(D_{2})}^{2}\leq Ch^{2}\left(
\left\Vert \phi _{1}(t)\right\Vert _{H^{2}(D_{1})}^{2}+\left\Vert \phi
_{2}(t)\right\Vert _{H^{2}(D_{2})}^{2}+h^{2}\left\Vert u(t)\right\Vert
_{H^{2}(D_{1})}^{2}\right) \\
\\
+C\left\Vert v_{s}(t)-v_{sh}(t)\right\Vert
_{L^{2}(D_{2})}^{2}+C\left\Vert
\theta _{u}(t)\right\Vert _{L^{2}(D_{1})}^{2}.%
\end{array}
\label{uv6}
\end{equation}
\end{lemma}

\begin{proof}
To calculate the estimate (\ref{uv6.0}) we set $v_{s}-v_{sh}=\left(
v_{s}-I_{h}^{0}v_{s}\right) +\left( I_{h}^{0}v_{s}-v_{sh}\right) $, so using
(\ref{pre4}) it follows that%
\begin{equation*}
\left\Vert v_{s}(t)-v_{sh}(t)\right\Vert
_{L^{2}(D_{2})}^{2}=Ch^{2}\left\vert v_{s}(t)\right\vert
_{H^{1}(D_{2})}^{2}+2\left\Vert I_{h}^{0}v_{s}(t)-v_{sh}(t)\right\Vert
_{L^{2}(D_{2})}^{2}.
\end{equation*}%
Since $I_{h}^{0}v_{s}-v_{sh}=\left( I_{h}^{0}v-v_{h\Delta r}\right)
(x;R_{s}(x),t)=\left( I_{0}^{x}\rho _{v}+\theta _{v}\right)
(x;R_{s}(x),t)$ (recalling the definition of $\theta_{v}$) then by
virtue of the definition of the $L^{2}$-norm for functions of
$V_{h}^{0}(\overline{D}_{2})$
presented in Section 3, we have that%
\begin{equation*}
\left\Vert I_{h}^{0}v_{s}(t)-v_{sh}(t)\right\Vert
_{L^{2}(D_{2})}^{2}\leq 2\sum_{l=1}^{M_{2}}\widehat{h}_{l}\left(
\rho _{v}^{2}(x_{l};R_{s}(x_{l}),t)+\theta
_{v}^{2}(x_{l};R_{s}(x_{l}),t)\right) .
\end{equation*}%
To estimate $\rho _{v}^{2}(x_{l};R_{s}(x_{l}),t)$ and $\theta
_{v}^{2}(x_{l};R_{s}(x_{l}),t)$ we make use of Lemma \ref{lem2}
noting that there exists a real number $a$, $0<a<R_{s}(x_{l})$, such
that $\rho _{v}^{2}(x_{l};R_{s}(x_{l}),t)\leq \left\Vert \rho
_{v}(x_{l};r,t)\right\Vert _{L^{\infty }(a,R(x_{l}))}^{2}$ and
$\theta _{v}^{2}(x_{l};R_{s}(x_{l}),t)\leq \left\Vert \theta
_{v}(x_{l};r,t)\right\Vert _{L^{\infty }(a,R(x_{l}))}^{2}$, thus by
virtue of (\ref{pre12}) it follows that there are a real number
$\epsilon$ and positive constant $C(\epsilon)>\epsilon$ such that
\begin{equation*}
\begin{array}{r}
\displaystyle\sum_{l=1}^{M_{2}}\widehat{h}_{l}\rho
_{v}^{2}(x_{l};R_{s}(x_{l}),t)\leq \displaystyle\epsilon \sum_{l=1}^{M_{2}}%
\widehat{h}_{l}\int_{0}^{R_{s}(x_{l})}r^{2}\left\vert \frac{\partial
\rho
_{v}(x_{l};r,t)}{\partial r}\right\vert ^{2}dr \\
\\
+\displaystyle C(\epsilon )\sum_{l=1}^{M_{2}}\widehat{h}%
_{l}\int_{0}^{R_{s}(x_{l})}r^{2}\left\vert \rho _{v}(x_{l};r,t)\right\vert
^{2}dr \\
\\
\leq C(\epsilon )\left\Vert I_{0}^{x}\rho _{v}(t)\right\Vert
_{L^{2}(D_{2},H_{r}^{1}(0,R_{s}(\cdot )))}^{2} \\
\\
\leq C(\epsilon )\left\Vert \rho _{v}(t)\right\Vert
_{L^{2}(D_{2},H_{r}^{1}(0,R_{s}(x)))} \\
\\
\leq C(\epsilon )\Delta r^{2}\left\Vert v(t)\right\Vert
_{L^{2}(D_{2},H_{r}^{2}(0,R_{s}(\cdot )))}^{2},%
\end{array}%
\end{equation*}%
because by approximation theory $\left\Vert I_{0}^{x}\rho
_{v}(t)\right\Vert _{L^{2}(D_{2},H_{r}^{1}(0,R_{s}(x)))}\leq
C\left\Vert \rho _{v}(t)\right\Vert
_{L^{2}(D_{2},H_{r}^{1}(0,R_{s}(x)))}$, this latter term
being estimated according to (\ref{pre10}). Similarly,%
\begin{equation}
\begin{array}{r}
\displaystyle\sum_{l=1}^{M_{2}}\widehat{h}_{l}\theta
_{v}^{2}(x_{l};R_{s}(x_{l}),t)\leq \epsilon \displaystyle\sum_{l=1}^{M_{2}}%
\widehat{h}_{l}\int_{0}^{R_{s}(x_{l})}r^{2}\left\vert \frac{\partial
\theta _{v}(x_{l};r,t)}{\partial r}\right\vert ^{2}dr \\
\\
+C(\epsilon )\displaystyle\sum_{l=1}^{M_{2}}\widetilde{h}_{l}%
\int_{0}^{R_{s}(x_{l})}r^{2}\left\vert \theta _{v}(x_{l};r,t)\right\vert
^{2}dr \\
\\
\leq \epsilon \left\vert \theta _{v}(t)\right\vert
_{L^{2}(D_{2},H_{r}^{1}(0,R_{s}(\cdot )))}^{2}+C(\epsilon )\left\Vert \theta
_{v}(t)\right\Vert _{L^{2}(D_{2},L_{r}^{2}(0,R_{s}(\cdot )))}^{2}.%
\end{array}
\label{theta}
\end{equation}%
So, putting these bounds together the result (\ref{uv6.0}) follows. To
calculate the estimate (\ref{uv6}) we notice that by virtue of (\ref{pre14})
and Theorem \ref{Theorem 1}%
\begin{equation*}
\begin{array}{r}
\left\Vert J-J_{h}\right\Vert _{L^{2}(D_{2})}^{2}\leq C\left(
h^{2}\left(\left\Vert \phi _{1}(t)\right\Vert
_{H^{2}(D_{1})}^{2}+\left\Vert \phi
_{2}(t)\right\Vert _{H^{2}(D_{2})}^{2}\right)\right. \\
\\
\left. +\left\Vert u(t)-u_{h}(t)\right\Vert _{L^{2}(D_{1})}^{2}+\left\Vert
v_{s}(t)-v_{sh}(t)\right\Vert _{L^{2}(D_{2})}^{2}\right) .%
\end{array}%
\end{equation*}%
Since $u-u_{h}=\rho _{u}+\theta _{u}$, then taking into account (\ref{pre5})
\begin{equation*}
\left\Vert u(t)-u_{h}(t)\right\Vert _{L^{2}(D_{1})}^{2}\leq Ch^{4}\left\Vert
u(t)\right\Vert _{H^{2}(D_{1})}^{2}+2\left\Vert \theta _{u}\right\Vert
_{L^{2}(D_{1})}^{2},
\end{equation*}%
so the result (\ref{uv6}) follows.
\end{proof}

Next, we calculate an estimate for $\theta _{v}$. To this end, we
obtain, based on equation (\ref{w2}), the integral equation for
$I_{0}^{x}v(x;r,t)$  that will be used for that purpose. Thus, for
each one of the mesh points $\{x_{l}\}_{l=1}^{M_{2}}$ of $D_{2h}$
the equation (\ref{w2}) reads%
\begin{equation}
\int_{0}^{R_{s}(x_{l})}\frac{\partial {v}^{(l)}}{\partial {t}}%
w^{(l)}r^{2}dr+\int_{0}^{R_{s}(x_{l})}k_{2}\frac{\partial v^{(l)}}{\partial r%
}\frac{\partial w^{(l)}}{\partial r}r^{2}dr=-\left(
R_{s}^{2}(x)a_{2}^{-1}(x)F^{-1}Jw^{(l)}(R_{s}(x))\right)|
_{x=x_{l}}, \label{I0v1}
\end{equation}%
where $v^{(l)}:=v(x_{l};r,t)\in H_{r}^{1}(0,R_{s}(x_{l})$ and
$w^{(l)}:=w(x_{l};r)\in H_{r}^{1}(0,R_{s}(x_{l}))$. Using the nodal
basis functions $\left\{ \chi _{l}(x)\right\} _{l=1}^{M_{2}} $ of
the finite element space $V_{h}^{(0)}(\overline{D}_{2})$ we can write%
\begin{equation*}
I_{0}^{x}{v(x;r,t)=}\sum_{l}^{M_{2}}v^{(l)}(r,t)\chi _{l}(x)\mathrm{\ and\ \
}w(x;r)=\sum_{l=1}^{M_{2}}w^{(l)}(r)\chi _{l}(x).
\end{equation*}%
Now, noting that
\begin{equation*}
\int_{D_{2}}\int_{0}^{R_{s}(x)}\left( I_{0}^{x}{v}\right)
wr^{2}drdx=\sum_{l=1}^{M_{2}}\int_{D_{2}}\chi _{l}^{2}(x)\left(
\int_{0}^{R_{s}(x)}v^{(l)}(r,t)w^{(l)}(r)r^{2}dr\right) dx,
\end{equation*}%
and for $x\in \widehat{e}_{l}$, $R_{s}(x)=R_{s}(x_{l})$, then it follows that%
\begin{equation*}
\sum_{l=1}^{M_{2}}\int_{D_{2}}\chi _{l}^{2}(x)\left(
\int_{0}^{R_{s}(x)}v^{(l)}(r,t)w^{(l)}(r)r^{2}dr\right) dx=\sum_{l=1}^{M_{2}}%
\widehat{h}_{l}\int_{0}^{R_{s}(x_{l})}v^{(l)}(r,t)w^{(l)}(r)r^{2}dr.
\end{equation*}%
Hence, (\ref{I0v1}) becomes
\begin{equation}
\int_{D_{2}}\int_{0}^{R_{s}(x)}\left( \frac{\partial I_{0}^{x}{v}}{\partial {%
t}}w+k_{2}\frac{\partial I_{0}^{x}v}{\partial r}\frac{\partial w}{\partial r}%
\right)
r^{2}drdx=-\int_{D_{2}}I_{0}^{x}(R_{s}^{2}(x)a_{2}^{-1}(x)F^{-1}Jw(x;R_{s}(x)))dx.
\label{uv4}
\end{equation}%
We proceed to formulate the equation for $\theta _{v}$. From (\ref{uv2}) it
follows that $v_{h\Delta r}=I_{0}^{x}v-(I_{0}^{x}\rho _{v}+\theta _{v})$,
then replacing this expression for $v_{h\Delta r}$ in (\ref{num2}) and using
(\ref{pre8}) and (\ref{uv4}) it follows that for all $w_{h\Delta r}\in
V_{h\Delta r}(\overline{D}_{3})$,%
\begin{equation}
\begin{array}{r}
\displaystyle\int_{D_{2}}\int_{0}^{R_{s}(x)}\left( \frac{\partial \theta _{v}%
}{\partial {t}}w_{h\Delta r}+k_{2}\frac{\partial \theta _{v}}{\partial r}%
\frac{\partial w_{h\Delta r}}{\partial r}\right) r^{2}drdx=\lambda
\int_{D_{2}}\int_{0}^{R_{s}(x)}I_{0}^{x}\rho _{v}w_{h\Delta r}r^{2}drdx \\
\\
-\displaystyle\int_{D_{2}}\int_{0}^{R_{s}(x)}\frac{\partial I_{0}^{x}\rho
_{v}}{\partial {t}}w_{h\Delta r}r^{2}drdx \\
\\
-\displaystyle\int_{D_{2}}I_{h}^{0}(R_{s}^{2}(x)a_{2}^{-1}(x)F^{-1}(J(x)-J_{h}(x))w_{sh}(x))dx
\\
\\
+\displaystyle\int_{D_{2}}\left(
R_{s}^{2}(x)a_{2}^{-1}(x)F^{-1}J_{h}(x)-I_{h}^{0}(R_{s}^{2}(x)a_{2}^{-1}F^{-1}(x)J_{h}(x)\right) )w_{sh}(x)dx,%
\end{array}
\label{uv5}
\end{equation}%
where we have made use of the following properties of the interpolant $%
I_{0}^{x}$: (i) for $w_{h\Delta r}(x,r)\in V_{h\Delta r}(\overline{D}_{3})$,
$w_{h\Delta r}(x,r)=I_{0}^{x}w_{h\Delta r}(x,r)$, and (ii) when $r=R_{s}(x)$%
, we can define the function $w_{sh}(x)=w_{h\Delta r}(x,R_{s}(x))$ such that
$w_{sh}(x)=I_{0}^{x}w_{h\Delta r}(x,R_{s}(x))=I_{h}^{0}w_{sh}(x)$.\ Setting $%
w_{h\Delta r}=\theta _{v}$ yields%
\begin{equation}
\begin{array}{l}
\displaystyle\frac{1}{2}\frac{d}{dt}\left\Vert \theta
_{v}(t)\right\Vert _{L^{2}(D_{2},L_{r}^{2}(0,R_{s}(\cdot
)))}^{2}+\displaystyle \underline{k_{2}}\left\Vert \frac{\partial
\theta _{v}(t)}{\partial r}\right\Vert
_{L^{2}(D_{2},L_{r}^{2}(0,R_{s}(\cdot )))}^{2} \\
\\
\leq \lambda \left\Vert I_{0}^{x}\rho _{v}(t)\right\Vert
_{L^{2}(D_{2},L_{r}^{2}(0,R_{s}(\cdot )))}\left\Vert \theta
_{v}(t)\right\Vert _{L^{2}(D_{2},L_{r}^{2}(0,R_{s}(\cdot )))} \\
\\
+\displaystyle\left\Vert \frac{\partial I_{0}^{x}\rho _{v}(t)}{\partial t}%
\right\Vert _{L^{2}(D_{2},L_{r}^{2}(0,R_{s}(\cdot )))}\left\Vert \theta
_{v}(t)\right\Vert _{L^{2}(D_{2},L_{r}^{2}(0,R_{s}(\cdot )))} \\
\\
+C\left\Vert I_{h}^{(0)}\left( (J-J_{h})\theta _{vs}(t)\right)
\right\Vert _{L^{1}(D_{2})}+C\left\Vert
(J_{h}-I_{(0)}^{0}J_{h})\theta _{vs}(t)\right\Vert
_{L^{1}(D_{2})}\equiv \sum_{i=1}^{4}R_{i},%
\end{array}
\label{uv7}
\end{equation}%
where,\ $\theta _{vs}(t)=\theta _{v}(x;R_{s}(x),t)$ is the value of $\theta
_{v}$ on the surface of the sphere of radius $R_{x}(x)$ associated with the
point $\{x\}$ of $D_{2}$.

\begin{lemma}
\label{lem7} There exists a constant $C$ independent of $h$ and
$\Delta r$, but depending on $\underline{k_{1}}$ and
$\underline{k_{2}}$, such that
\begin{equation}
\begin{array}{c}
\displaystyle\frac{d}{dt}\left\Vert \theta _{v}(t)\right\Vert
_{L^{2}(D_{2},L_{r}^{2}(0,R_{s}(\cdot )))}^{2}+\displaystyle
\underline{k_{2}}\left\Vert \frac{\partial \theta _{v}(t)}{\partial
r}\right\Vert
_{L^{2}(D_{2},L_{r}^{2}(0,R_{s}(\cdot )))}^{2} \\
\\
\leq Ch^{2}\left( \left\Vert \phi _{1}(t)\right\Vert
_{H^{2}(D_{1})}^{2}+\left\Vert \phi _{2}(t)\right\Vert
_{H^{2}(D_{2})}^{2}\right) \\
\\
+Ch^{2}\left( \left\vert v_{s}(t)\right\vert
_{H^{1}(D_{2})}^{2}+h^{2}\left\Vert u(t)\right\Vert _{H^{2}(D_{1})}^{2}+%
\displaystyle\left\Vert \frac{\partial J}{\partial x}\right\Vert
_{L^{2}(D_{2})}^{2}\right) \\
\\
+C\Delta r^{2}\left( \left\Vert v(t)\right\Vert
_{L^{2}(D_{2},H_{r}^{2}(0,R_{s}(\cdot )))}^{2}+\Delta r^{2}\displaystyle%
\left\Vert \frac{\partial v(t)}{\partial t}\right\Vert
_{L^{2}(D_{2},H_{r}^{2}(0,R_{s}(\cdot )))}^{2}\right) \\
\\
+C\left(\left\Vert \theta _{u}(t)\right\Vert
_{L^{2}(D_{1})}^{2}+\left\Vert \theta _{v}(t)\right\Vert
_{L^{2}(D_{2},L_{r}^{2}(0,R_{s}(\cdot )))}\right).%
\end{array}
\label{uv7.1}
\end{equation}
\end{lemma}

\begin{proof}
We bound the terms $R_{1},\ldots ,R_{4}$ on the right hand side of (\ref{uv7}%
). Noting that $\left\Vert I_{0}^{x}\rho _{v}(t)\right\Vert
_{L^{2}(D_{2},L_{r}^{2}(0,R_{s}(\cdot )))}\leq C\left\Vert \rho
_{v}(t)\right\Vert _{L^{2}(D_{2},L_{r}^{2}(0,R_{s}(\cdot )))}$, then by
virtue of Young inequality and (\ref{pre10}) it follows that%
\begin{equation}
R_{1}\leq C\Delta r^{4}\left\Vert v(t)\right\Vert
_{L^{2}(D_{2},H_{r}^{2}(0,R_{s}(\cdot )))}^{2}+C\left\Vert \theta
_{v}(t)\right\Vert _{L^{2}(D_{2},L_{r}^{2}(0,R_{s}(\cdot )))}^{2}.
\label{uv8}
\end{equation}%
To bound $R_{2}$ we notice that%
\begin{equation*}
\begin{array}{r}
\displaystyle\left\Vert \frac{\partial I_{0}^{x}\rho _{v}(t)}{\partial t}%
\right\Vert _{L^{2}(D_{2},L_{r}^{2}(0,R_{s}(\cdot )))}\leq C\displaystyle%
\left\Vert \frac{\partial \rho _{v}(t)}{\partial t}\right\Vert
_{L^{2}(D_{2},L_{r}^{2}(0,R_{s}(\cdot )))} \\
\\
= C\displaystyle\left\Vert (I-P_{1}^{r})\frac{\partial v(t)}{\partial t}%
\right\Vert _{L^{2}(D_{2},L_{r}^{2}(0,R_{s}(\cdot )))} \\
\\
\leq C\Delta r^{2}\displaystyle\left\Vert \frac{\partial v(t)}{\partial t}%
\right\Vert _{L^{2}(D_{2},H_{r}^{2}(0,R_{s}(\cdot )))}.%
\end{array}%
\end{equation*}%
Hence, by Young inequality it follows that%
\begin{equation}
R_{2}\leq C\Delta r^{4}\left\Vert \frac{\partial v(t)}{\partial t}%
\right\Vert _{L^{2}(D_{2},H_{r}^{2}(0,R_{s}(\cdot )))}^{2}+C\left\Vert
\theta _{v}(t)\right\Vert _{L^{2}(D_{2},L_{r}^{2}(0,R_{s}(\cdot )))}^{2}.
\label{uv9}
\end{equation}%
We bound the term $R_{3}$. Thus, we have that%
\begin{equation*}
\begin{array}{l}
C\left\Vert I_{h}^{0}\left( (J-J_{h})\theta _{vs}(t)\right)
\right\Vert
_{L^{1}(D_{2})}=C\displaystyle\sum_{l=1}^{M_{2}}\widehat{h}_{l}\left\vert
\left( J(x_{l})-J_{h}(x_{l}\right) )\theta
_{v}(x_{l};R_{s}(x_{l}),t)\right\vert \\
\\
\leq
\displaystyle\frac{C^{2}}{2}\sum_{l=1}^{M_{2}}\widehat{h}_{l}\left(
J(x_{l})-J_{h}(x_{l}\right) )^{2}+\displaystyle\frac{1}{2}\sum_{l=1}^{M_{2}}%
\widehat{h}_{l}\theta _{v}^{2}(x_{l};R_{s}(x_{l}),t) \\
\\
\leq \displaystyle\frac{C^{2}}{2}\left\Vert
I_{h}^{0}(J-J_{h})\right\Vert
_{L^{2}(D_{2})}^{2}+\displaystyle\frac{1}{2}\sum_{l=1}^{M_{2}}\widehat{h}%
_{l}\theta _{v}^{2}(x_{l};R_{s}(x_{l}),t).%
\end{array}%
\end{equation*}%
Estimating the last term on the right hand side of this inequality as we did
before in the proof of Lemma \ref{lem6}, see (\ref{theta}), and noting that $%
\left\Vert I_{h}^{0}(J-J_{h})\right\Vert _{L^{2}(D_{2})}\leq C\left\Vert
J-J_{h}\right\Vert _{L^{2}(D_{2})}^{2}$, it readily follows\textbf{\ }that%
\begin{equation}
R_{3}\leq C\left\Vert J-J_{h}\right\Vert _{L^{2}(D_{2})}^{2}+C(\epsilon
)\left\Vert \theta _{v}(t)\right\Vert _{L^{2}(D_{2},L_{r}^{2}(0,R_{s}(\cdot
)))}^{2}+\epsilon \left\Vert \frac{\partial \theta _{v}(t)}{\partial r}%
\right\Vert _{L^{2}(D_{2},L_{r}^{2}(0,R_{s}(\cdot )))}^{2}.  \label{uv10}
\end{equation}%
To bound $R_{4}$ we notice that $
J_{h}-I_{h}^{(0)}J_{h}=(J_{h}-J)+(J-I_{h}^{(0)}J)+I_{h}^{(0)}(J-J_{h})$,
so by the triangle inequality it follows that%
\begin{equation*}
\begin{array}{l}
C\left\Vert (J_{h}-I_{h}^{(0)}J_{h})\theta _{vs}(t)\right\Vert
_{L^{1}(D_{2})}\leq C\left\Vert (J_{h}-J)\theta _{vs}(t)\right\Vert
_{L^{1}(D_{2})} \\
\\
+C\left\Vert (J-I_{h}^{(0)}J)\theta _{vs}(t)\right\Vert
_{L^{1}(D_{2})}+C\left\Vert I_{h}^{(0)}(J-J_{h})\theta
_{vs}(t)\right\Vert
_{L^{1}(D_{2})}.%
\end{array}%
\end{equation*}%
Noting that $\left\Vert ab\right\Vert _{L^{1}(D_{2})}\leq \frac{\varepsilon
}{2}\left\Vert a\right\Vert _{L^{2}(D_{2})}^{2}+\frac{1}{2\varepsilon }%
\left\Vert b\right\Vert _{L^{2}(D_{2})}^{2}$ and applying the same argument
as we have just done to bound $R_{3}$, we obtain that%
\begin{equation*}
\begin{array}{r}
\left\Vert (J_{h}-I_{h}^{0}J_{h})\theta _{vs}(t)\right\Vert
_{L^{1}(D_{2})}\leq C\left( \left\Vert J-J_{h}\right\Vert
_{L^{2}(D_{2})}^{2}+\left\Vert J-I_{h}^{0}J\right\Vert
_{L^{2}(D_{2})}^{2}\right) \\
\\
+\sum_{l=1}^{M_{2}}\widehat{h}_{l}\theta
_{v}^{2}(x_{l};R_{s}(x_{l}),t).%
\end{array}%
\end{equation*}%
We bound the last term of this inequality as we have done for
$R_{3},$ and by virtue of assumption \textbf{A3 }set $\left\Vert
J-I_{h}^{(0)}J\right\Vert
_{L^{2}(D_{2})}^{2}\leq Ch^{2}\left\Vert \frac{\partial J}{\partial x}%
\right\Vert _{L^{2}(D_{2})}^{2}$. Hence,
\begin{equation}
\begin{array}{r}
R_{4}\leq C(\epsilon )\left\Vert \theta _{v}(t)\right\Vert
_{L^{2}(D_{2},L_{r}^{2}(0,R_{s}(\cdot )))}^{2}+\epsilon \displaystyle%
\left\vert \frac{\partial \theta _{v}(t)}{\partial r}\right\vert
_{L^{2}(D_{2},L_{r}^{2}(0,R_{s}(\cdot )))}^{2} \\
\\
+C\left\Vert J-J_{h}\right\Vert _{L^{2}(D_{2})}^{2}+Ch^{2}\displaystyle%
\left\Vert \frac{\partial J}{\partial x}\right\Vert _{L^{2}(D_{2})}^{2}.%
\end{array}
\label{uv10_1}
\end{equation}%
Letting $\epsilon=\underline{k_{2}}/8$ in (\ref{uv6}), (\ref{uv10})
and (\ref{uv10_1}),
and replacing (\ref{uv8})-(\ref{uv10_1}) in (\ref{uv7}), the result (\ref%
{uv7.1}) follows.
\end{proof}

Next, we proceed to calculate an estimate for $\theta _{u}$. Thus,
subtracting (\ref{num1}) from (\ref{w1}) it readily follows that%
\begin{equation*}
\int_{D_{1}}\frac{\partial \theta _{u}}{\partial {t}}w_{h}dx+%
\int_{D_{1}}k_{1}\frac{\partial \theta _{u}}{\partial x}\frac{dw_{h}}{dx}%
dx=\lambda \int_{D_{1}}\rho _{u}w_{h}dx-\int_{D_{1}}\frac{\partial \rho _{u}%
}{\partial {t}}w_{h}dx+\int_{D_{1}}a_{1}\left( J-J_{h}\right) w_{h}dx.
\end{equation*}%
Setting $w_{h}=\theta _{u}$ in this equation yields%
\begin{equation*}
\begin{array}{r}
\displaystyle\frac{1}{2}\frac{d}{dt}\left\Vert \theta
_{u}(t)\right\Vert
_{L^{2}(D_{1})}^{2}+\underline{k_{1}}\displaystyle\left\Vert
\frac{\partial \theta _{u}(t)}{\partial x}\right\Vert
_{L^{2}(D_{1})}^{2}\leq C\displaystyle\left(
\left\Vert \rho _{u}(t)\right\Vert _{L^{2}(D_{1})}^{2}+\left\Vert \frac{%
\partial \rho _{u}(t)}{\partial t}\right\Vert _{L^{2}(D_{1})}^{2}\right) \\
\\
+C\left\Vert J-J_{h}\right\Vert _{L^{2}(D_{1})}^{2}+C\left\Vert \theta
_{u}(t)\right\Vert _{L^{2}(D_{1})}^{2}.%
\end{array}%
\end{equation*}%
By virtue of Lemma \ref{lem6} we have the following result.

\begin{lemma}
\label{lem8} There exists a constant $C$ independent of $h$ and
$\Delta r$, but depending on $\underline{k_{1}}$ and
$\underline{k_{2}}$, such that
\begin{equation}
\begin{array}{l}
\displaystyle\frac{d}{dt}\left\Vert \theta _{u}(t)\right\Vert
_{L^{2}(D_{1})}^{2}+\displaystyle \underline{k_{1}}\left\Vert
\frac{\partial \theta
_{u}(t)}{\partial x}\right\Vert _{L^{2}(D_{1})}^{2} \\
\\
\leq C\left\{ h^{2}(\left\Vert \phi _{1}(t)\right\Vert
_{H^{2}(D_{1})}^{2}+\left\Vert \phi _{2}(t)\right\Vert
_{H^{2}(D_{2})}^{2}+\left\vert v_{s}(t)\right\vert
_{H^{1}(D_{2})}^{2})\right. \\
\\
\left. +h^{4}\left( \left\Vert u(t)\right\Vert
_{H^{2}(D_{1})}^{2}+\displaystyle\left\Vert \frac{\partial
u(t)}{\partial t}\right\Vert _{H^{2}(D_{1})}^{2}\right) +\Delta
r^{2}\left\Vert v(t)\right\Vert
_{L^{2}(D_{2},H_{r}^{2}(0,R_{s}(\cdot )))}^{2}\right\} \\
\\
+C\left( \left\Vert \theta _{u}(t)\right\Vert
_{L^{2}(D_{1})}^{2}+\left\Vert \theta _{v}(t)\right\Vert
_{L^{2}(D_{2},L_{r}^{2}(0,R_{s}(\cdot )))}\right)+\displaystyle
\frac{\underline{k_{2}}}{4} \left\Vert \frac{\partial {\theta
_{v}(t)}}{\partial r} \right\Vert
 _{L^{2}(D_{2},L_{r}^{2}(0,R_{s}(\cdot )))}.%
\end{array}
\label{uv11}
\end{equation}
\end{lemma}

We are now in a position to establish the main result of this subsection.

\begin{theorem}
\label{Theorem 2} Let $(u,v,\phi _{1},\phi _{2})$ and $(u_{h},v_{h\Delta
r},\phi _{1h},\phi _{2h})$ be the solutions to (\ref{w1})-(\ref{w4}) and (%
\ref{num1})-(\ref{num4}) respectively, with%
\begin{equation*}
\left\Vert u(0)-u_{h}(0)\right\Vert _{L^{2}(D_{1})}\leq Ch^{2}\ \mathrm{and\
}\left\Vert v(0)-v_{h\Delta r}(0)\right\Vert
_{L^{2}(D_{2},L_{r}^{2}(0,R_{s}(\cdot )))}\leq C(h+\Delta r^{2}).
\end{equation*}%
Furthermore, for $0\leq t\leq T_{\mathrm{end}},$ the following regularity
assumptions hold:

R1)$\ u$ and $\displaystyle \frac{\partial u}{\partial t}\in L^{2}(0,T_{%
\mathrm{end}};H^{2}(D_{1})),$

\bigskip

R2) $v$ and $\displaystyle \frac{\partial v(t)}{\partial t}\in L^{2}(0,T_{%
\mathrm{end}};L^{2}(D_{2};H_{r}^{2}(0,R_{s}(\cdot )))),$ and $v_{s}\in
L^{2}(0,T_{\mathrm{end}};H^{1}(D_{2})),$

\bigskip

R3) $\phi _{1}\in L^{2}(0,T_{\mathrm{end}};H^{2}(D_{1}))$, $\phi _{2}\in
L^{2}(0,T_{\mathrm{end}};H^{2}(D_{2})),$ and $J\in L^{2}(0,T_{\mathrm{end}%
};H^{1}(D_{2}));$

\bigskip

then there is a constant $C(t,\underline{k_{1}},\underline{k_{2}})$ such that%
\begin{equation}
\left\Vert u(t)-u_{h}(t)\right\Vert _{L^{2}(D_{1})}^{2}+\left\Vert
v(t)-v_{h\Delta r}(t)\right\Vert _{L^{2}(D_{2},L_{r}^{2}(0,R_{s}(\cdot
)))}^{2}+\int_{0}^{t}\left\Vert \Phi (\tau )-\Phi _{h}(\tau )\right\Vert
_{V}^{2}d\tau \leq C(h^{2}+\Delta r^{2}).  \label{uv12}
\end{equation}
\end{theorem}

\begin{proof}
Bounding $\left\Vert u(t)-u_{h}(t)\right\Vert _{L^{2}(D_{1})}^{2}$ as $%
2\left( \left\Vert \rho _{u}(t)\right\Vert ^{2}+\left\Vert \theta
_{u}(t)\right\Vert ^{2}\right) $ and using the estimate of Lemma \ref{lem6}
for $\left\Vert v_{s}(t)-v_{sh}(t)\right\Vert _{L^{2}(D_{2})}$, Theorem \ref%
{Theorem 1} yields%
\begin{equation*}
\begin{array}{l}
\left\Vert \Phi (t)-\Phi _{h}(t)\right\Vert _{V}^{2}\leq Ch^{2}\left\vert
v_{s}(t)\right\vert _{H^{1}(D_{2})}^{2}+Ch^{4}\left\Vert u(t)\right\Vert
_{H^{2}(D_{1})}^{2}+C\left\Vert \theta _{u}\right\Vert _{L^{2}(D_{1})}^{2}
\\
\\
+C(\epsilon )\left\Vert \theta _{v}(t)\right\Vert
_{L^{2}(D_{2},L_{r}^{2}(0,R_{s}(\cdot )))}^{2}+\epsilon \displaystyle%
\left\Vert \frac{\partial \theta _{v}(t)}{\partial r}\right\Vert
_{L^{2}(D_{2},L_{r}^{2}(0,R_{s}(\cdot )))}^{2}.%
\end{array}%
\end{equation*}%
Considering this estimate together with those of Lemmas \ref{lem7} and \ref%
{lem8} it readily follows that%
\begin{equation*}
\begin{array}{r}
\displaystyle\frac{d}{dt}\left( \left\Vert \theta _{u}(t)\right\Vert
_{L^{2}(D_{1})}^{2}+\left\Vert \theta _{v}(t)\right\Vert
_{L^{2}(D_{2},L_{r}^{2}(0,R_{s}(\cdot )))}^{2}\right) +\left\Vert \Phi
(t)-\Phi _{h}(t)\right\Vert _{V}^{2}\leq C(h^{2}+\Delta r^{2}) \\
\\
+C\left( \left\Vert \theta _{u}(t)\right\Vert
_{L^{2}(D_{1})}^{2}+\left\Vert \theta _{v}(t)\right\Vert
_{L^{2}(D_{2},L_{r}^{2}(0,R_{s}(\cdot )))}^{2}\right) .%
\end{array}%
\end{equation*}%
Then, Gronwall inequality yields%
\begin{equation}
\left\Vert \theta _{u}(t)\right\Vert _{L^{2}(D_{1})}^{2}+\left\Vert
\theta _{v}(t)\right\Vert _{L^{2}(D_{2},L_{r}^{2}(0,R_{s}(\cdot
)))}^{2}+\int_{0}^{t}\left\Vert \Phi (\tau )-\Phi _{h}(\tau
)\right\Vert _{V}^{2}d\tau \leq
C(t,\underline{k_{1}},\underline{k_{2}})(h^{2}+\Delta r^{2}).
\label{uv13}
\end{equation}%
The terms $\left\Vert \theta _{u}(0)\right\Vert _{L^{2}(D_{1})}^{2}$ and $%
\left\Vert \theta _{v}(0)\right\Vert
_{L^{2}(D_{2},L_{r}^{2}(0,R_{s}(\cdot )))}^{2}$ are considered to be
zero. From (\ref{uv1}), (\ref{uv2}), (\ref{uv2_1}) and (\ref{uv3}) it follows that%
\begin{equation*}
\begin{array}{r}
\left\Vert u(t)-u_{h}(t)\right\Vert _{L^{2}(D_{1})}^{2}+\left\Vert
v(t)-v_{h\Delta r}(t)\right\Vert _{L^{2}(D_{2},L_{r}^{2}(0,R_{s}(\cdot
)))}^{2}\leq 2\left( \left\Vert \rho _{u}(t)\right\Vert
_{L^{2}(D_{1})}+\left\Vert \theta _{u}(t)\right\Vert _{L^{2}(D_{1})}\right)
\\
\\
+C\left( h^{2}\left\Vert
v(t)\right\Vert_{H^{1}(D_{2};L_{r}^{2}(0,R_{s}(\cdot)))} +\left\Vert
\rho _{v}(t)\right\Vert _{L^{2}(D_{2},L_{r}^{2}(0,R_{s}(\cdot
))}^{2}+\left\Vert \theta
_{v}(t)\right\Vert _{L^{2}(D_{2},L_{r}^{2}(0,R_{s}(\cdot ))}^{2}\right) .%
\end{array}%
\end{equation*}%
So, by combining this inequality with (\ref{uv13}) we obtain (\ref{uv12}).
\end{proof}

\section{Fully discrete model}

We consider now the fully discrete model based on the time stepping backward
Euler scheme. This scheme has been used to discretize in time the equations
of the P2D model either with finite differences \cite{dual}, finite volumes
\cite{SW}, \cite{ML}, or finite elements. For convenience, hereafter we
shall use the notation $a^{n}:=a(x,t_{n})$, where $n$ is a nonnegative
integer and $t_{n}=n\Delta t$, $\Delta t$ being a uniform time step. The
formulation of the fully discrete model is as follows. Assuming that at time
$t_{n-1}$,\ $n=1,2,\ldots ,N$, the solution $(u_{h}^{n-1},v_{h\Delta
r}^{n-1},\phi _{1h}^{n-1},\phi _{1h}^{n-1})\in V_{h}^{(1)}(\overline{D}%
_{1})\times V_{h\Delta r}(\overline{D}_{3})\times W_{h}(\overline{D}%
_{1})\times V_{h}^{(1)}(\overline{D}_{2})$ is known, calculate $%
(u_{h}^{n},v_{h\Delta r}^{n},\phi _{1h}^{n},\phi _{1h}^{n})\in V_{h}^{(1)}(%
\overline{D}_{1})\times V_{h\Delta r}(\overline{D}_{3})\times W_{h}(%
\overline{D}_{1})\times V_{h}^{(1)}(\overline{D}_{2})$ as solution of the
system%
\begin{equation}
\int_{D_{1}}\widetilde{\partial}
_{t}u^{n}{_{h}}w_{h}dx+\int_{D_{1}}k_{1}\frac{du_{h}^{n}}{dx}\frac{dw_{h}}{dx}dx=\int_{D_{1}}a_{1}J_{h}^{n}w_{h}dx\
\forall w_{h}\in V_{h}^{(1)}(\overline{D}_{1}).  \label{fdm1}
\end{equation}

\begin{equation}
\left\{
\begin{array}{l}
\displaystyle\int_{D_{2}}\int_{0}^{R_{s}(x)}\displaystyle\widetilde{\partial}
_{t}v^{n}_{h\Delta r}w_{h\Delta r}r^{2}drdx+\displaystyle\int_{D_{2}}%
\int_{0}^{R_{s}(x)}k_{2}\displaystyle\frac{\partial v^{n}{_{h\Delta r}}}{%
\partial {r}}\frac{\partial w_{h\Delta r}}{\partial r}r^{2}drdx \\
\\
=-\displaystyle\int_{D_{2}}\frac{R_{s}^{2}(x)J_{h}^{n}w_{h\Delta
r}(x,R_{s}(x))}{a_{2}(x)F}dx\ \ \forall w_{h\Delta r}\in V_{h\Delta r}(%
\overline{D}_{3}).%
\end{array}%
\right.  \label{fdm2}
\end{equation}

\begin{equation}
\int_{D_{1}}\kappa (u_{h}^{n})\frac{d\phi _{1h}^{n}}{dx}%
\frac{dw_{h}}{dx}dx=\int_{D_{1}}J_{h}^{n}w_{h}dx\ \ \forall w_{h}\in
V_{h}^{(1)}(\overline{D}_{1}).  \label{fdm3}
\end{equation}

\bigskip

\begin{equation}
\int_{D_{2}}\sigma \frac{d\phi _{2h}^{n}}{dx}\frac{dw_{h}}{dx%
}dx=-\int_{D_{2}}\left( J_{h}^{n}+g\right) w_{h}dx\ \ \forall w_{h}\in
V_{h}^{(1)}(\overline{D}_{2}).  \label{fdm4}
\end{equation}

\bigskip

\begin{equation}
\int_{D_{2}}J_{h}^{n}dx=0\text{ with }\int_{D_{\mathrm{a}%
}}J_{h}^{n}dx=I(t_{n})=-\int_{D_{\text{\textrm{c}}}}J_{h}^{n}dx,
\label{fdm5}
\end{equation}%
where

\begin{equation}
\begin{array}{l}
\widetilde{\partial}
_{t}u_{h}^{n}=\displaystyle\frac{u_{h}^{n}-u_{h}^{n-1}}{\Delta t},
\text{\textrm{\ \ }}\widetilde{\partial}
_{t}v_{h\Delta r}^{n}=\displaystyle\frac{%
v_{h\Delta r}^{n}-v_{h\Delta r}^{n-1}}{\Delta t}, \\
\\
J_{h}^{n}=J(x,u_{_{h}}^{n},v_{sh}^{n},\eta _{h}^{n})=a_{2}(x)i_{0h}^{n}\sinh
\left( \beta \eta _{h}^{n}\right),\ i_{0h}^{n}=i_{0}(u_{h}^{n},v_{sh}^{n}) , \\
\\
\eta _{h}^{n}=\phi _{1h}^{n}-\phi _{2h}^{n}-\alpha^{n}_{h} \ln u_{h}^{n}-\overline{U}%
_{h}(v_{sh}^{n}),\ \alpha^{n}_{h}=\alpha(u_{h}^{n}).%
\end{array}
\label{fdm6}
\end{equation}

\subsection{On the existence and uniqueness of the solution of the fully
discrete model}

To prove that the system (\ref{fdm1})-(\ref{fdm4}) has a unique
solution, we first show that assuming $(u_{h}^{n},v_{h\Delta
r}^{n},v_{sh}^{n})\in V_{h}^{(1)}(\overline{D}_{1})\times V_{h\Delta
r}(\overline{D}_{3})\times V_{h}^{(0)}(\overline{D}_{2})$ and the
assumptions \textbf{A1-A4} hold, the system
(\ref{fdm3})-(\ref{fdm4}) has a unique solution $(\phi
_{1h}^{n},\phi _{2h}^{n})\in W_{h}(\overline{D}_{1})\times
V_{h}^{(1)}(\overline{D}_{2})$; then, returning to the system
(\ref{fdm1})-(\ref{fdm2}) and applying a well-known \ consequence of
Brower\'{}s fixed point theorem, which is
presented as Corollary 1.1 in \cite{GR}, we prove that there exists $%
(u_{h}^{n},v_{h\Delta r}^{n})\in V_{h}^{(1)}(\overline{D}_{1})\times
V_{h\Delta r}(\overline{D}_{3})$.

\begin{lemma}
\label{lem9} Assuming that for all $n$, $(u_{h}^{n},v_{h\Delta
r}^{n},v_{sh}^{n})\in V_{h}^{(1)}(\overline{D}_{1})\times V_{h\Delta r}(%
\overline{D}_{3})\times V_{h}^{(0)}(\overline{D}_{2})$, and the
assumptions \textbf{A1-A4} hold, then the system
(\ref{fdm3})-(\ref{fdm4}) has a unique solution $(\phi
_{1h}^{n},\phi _{2h}^{n})\in W_{h}(\overline{D}_{1})\times
V_{h}^{(1)}(\overline{D}_{2})$.
\end{lemma}

\begin{proof}
Looking at (\ref{fdm3})-(\ref{fdm4}) and in order to apply
Minty-Browder theorem to prove
the existence of a solution, we define\ the functions%
\begin{equation*}
\widehat{\eta }_{h}^{n}:=-\alpha_{h}^{n} \ln u_{h}^{n}-\overline{U}
_{h}(v_{sh}^{n})\ \mathrm{and\ }\widehat{J}%
_{h}^{n}:=J_{h}^{n}(x,u_{h}^{n},v_{sh}^{n},\widehat{\eta }_{h}^{n}),
\end{equation*}%
it is worth noticing that $\widehat{\eta }_{h}^{n}$ is equal to
$\eta _{h}^{n}$ when the potentials $\phi _{1h}^{n}$ and $\phi
_{1h}^{n}$ are zero. Now, going back to Section 4.2 and using
$\widehat{J}_{h}^{n}$, we
define the operators $\widehat{B}_{h}:V_{h}\rightarrow V_{h}^{\ast }$ and $%
A_{h}:V_{h}\rightarrow V_{h}^{\ast }$ as follows: for all
$n=1,2,..,N$,
\begin{equation*}
\left\langle \widehat{B}_{h}(\Phi _{h}^{n}),\Psi _{h}\right\rangle
=\int_{D_{2}}(J_{h}^{n}-\widehat{J}_{h}^{n})(\psi _{2h}-\psi
_{1h})dx\ \ \forall \Psi _{h}\in V_{h}.
\end{equation*}%
and%
\begin{equation*}
\left\langle A_{h}(\Phi _{h}^{n}),\Psi _{h}\right\rangle =a_{h}(\Phi
_{h}^{n},\Psi _{h})+\left\langle \widehat{B}_{h}(\Phi _{h}^{n}),\Psi
_{h}\right\rangle .
\end{equation*}%
Notice that when $\Phi _{h}^{n}=(0,0)$, $\left\langle
\widehat{B}_{h}(\Phi _{h}^{n}),\Psi _{h}\right\rangle =0$ because
$J_{h}^{n}=\widehat{J}_{h}^{n}$. Now, we can recast
(\ref{fdm3})-(\ref{fdm4}) as
follows. Find $\Phi _{h}^{n}:=(\phi _{1}^{n},\phi _{1}^{n})\in W_{h}(%
\overline{D}_{1})\times V_{h}^{(1)}(\overline{D}_{2})$ such that%
\begin{equation}
\left\langle A_{h}(\Phi _{h}^{n}),\Psi _{h}\right\rangle
=-\int_{D_{2}}g\psi _{2h}dx-\int_{D_{2}}\widehat{J}_{h}^{n}(\psi
_{2h}-\psi _{1h})dx\ \ \forall \Psi _{h}\in V_{h}.  \label{opA}
\end{equation}%
We can prove, using the same arguments as in Lemma \ref{lem5}, that the operator $%
\widehat{B}_{h}$ is monotone, bounded and continuous satisfying an
inequality as (\ref{pre17}); since the bilinear form $a_{h}$ is
continuous and semi-definite positive, then it follows that the
operator $A_{h}$ is monotone, bounded and continuous satisfying an
inequality as (\ref{pre17}). In order to prove that (\ref{opA}) has
a solution it remains to show that $A_{h}$ is coercive, i.e.,
$\forall \Phi _{h}^{n}\in W_{h}(\overline{D}_{1})\times
V_{h}^{(1)}(\overline{D}_{2})$,
there exists a positive constant $\alpha $ such that%
\begin{equation*}
\left\langle A_{h}(\Phi _{h}^{n}),\Phi _{h}^{n}\right\rangle \geq
\alpha \left\Vert \Phi _{h}^{n}\right\Vert _{V}^{2}.
\end{equation*}%
This can be easily done by considering the following facts: 1)\
$A_{h}$ is monotone; 2) it is easy to check, using the same
arguments as in Theorem \ref{Theorem 1} to prove (\ref{eep4}), that
$\forall \Phi _{h}^{n}$, $\overline{\Phi }_{h}^{n}$ $\in
W_{h}(\overline{D}_{1})\times V_{h}^{(1)}(\overline{D}_{2})$%
\begin{equation*}
a_{h}(\Phi _{h}^{n}-\overline{\Phi }_{h}^{n},\Phi _{h}^{n}-\overline{\Phi }%
_{h}^{n})+\left\langle \widehat{B}_{h}(\Phi _{h}^{n})-\widehat{B}_{h}(%
\overline{\Phi }_{h}^{n}),\Phi _{h}^{n}-\overline{\Phi }_{h}^{n}\right%
\rangle \geq \alpha \left\Vert \Phi _{h}^{n}-\overline{\Phi }%
_{h}^{n}\right\Vert _{V}^{2},
\end{equation*}%
then taking$\ \overline{\Phi }_{h}^{n}=(0,0)$ it follows the
coerciveness of $A_{h}$. Hence, the Minty-Browder theorem
\cite{Zeid} guaranties the existence of a solution $\Phi _{h}^{n}$
of (\ref{opA}). To prove the uniqueness of this
solution we follow the argument put forward in \cite{WXZ} to prove the uniqueness of the exact solution,
and assume that there two solutions $%
\Phi _{h}^{n}:=(\phi _{1h}^{n},\phi _{2h}^{n})$ and $\overline{\Phi }%
_{h}^{n}:=\left( \overline{\phi }_{1h}^{n},\overline{\phi
}_{2h}^{n}\right) $ of (\ref{fdm3})-(\ref{fdm4}), then setting,
$z_{1h}=\phi _{1h}^{n}-\overline{\phi }_{1h}^{n}$ and $z_{2h}=\phi
_{2h}^{n}-\overline{\phi }_{2h}^{n}$,\ from (\ref{fdm3})\ it
follows that%
\begin{equation*}
\int_{D_{1}}\kappa (u_{h}^{n})\frac{dz_{1h}}{dx}\frac{dw_{h}}{dx}%
dx=\int_{D_{2}}\left( J_{h}^{n}-\overline{J}_{h}^{n}\right) w_{h}dx\
\ \forall w_{h}\in V_{h}^{\left( 1\right) }(\overline{D}_{1})
\end{equation*}%
and from (\ref{fdm4})%
\begin{equation*}
\int_{D_{1}}\sigma
\frac{dz_{2h}}{dx}\frac{dv_{h}}{dx}dx=-\int_{D_{2}}\left(
J_{h}^{n}-\overline{J}_{h}^{n}\right) v_{h}dx\ \ \forall v_{h}\in
V_{h}^{\left( 1\right) }(\overline{D}_{2}),
\end{equation*}%
where\ $J_{h}^{n}=J_{h}^{n}(x,u_{h}^{n},v_{sh}^{n},\eta _{h}^{n})$ and $%
\overline{J}_{h}^{n}=J_{h}^{n}(x,u_{h}^{n},v_{sh}^{n},\overline{\eta }%
_{h}^{n})$, with $\overline{\eta }_{h}^{n}=\overline{\phi }_{2h}^{n}-%
\overline{\phi }_{1h}^{n}-\alpha_{h}^{n} \ln u_{h}^{n}-\overline{U}
_{h}(v_{sh}^{n}) $. Setting $w_{h}=z_{1h}$ and $v_{h}=z_{2h}$ and
applying the mean value
theorem one readily obtains that%
\begin{equation*}
\int_{D_{1}}\kappa (u_{h}^{n})\left( \frac{dz_{1h}}{dx}\right)
^{2}dx+\int_{D_{1}}\sigma \left( \frac{dz_{2h}}{dx}\right)
^{2}dx+\int_{D_{2}}\frac{\partial J_{h}^{n}\left( \xi \right)
}{\partial \eta^{n} }(z_{2h}-z_{1h})^{2}=0,
\end{equation*}
here $\frac{\partial J_{h}^{n}\left( \xi \right) }{\partial \eta^{n}
}>0$ according to assumption \textbf{A3}. The first term of this
expression
implies that for all $n$,%
\begin{equation*}
z_{1h}=\phi _{1h}^{n}-\overline{\phi }_{1h}^{n}=K_{1},
\end{equation*}%
but the constant $K_{1}=0$ because $\phi _{1h}^{n}$ and $\overline{\phi }%
_{1h}^{n}$ are in $W_{h}(\overline{D}_{1})$, so $\phi _{1h}^{n}=\overline{%
\phi }_{1h}^{n}$ Similarly, from the second and third terms it follows that $%
z_{2h}=0$, and consequently $\phi _{2h}^{n}=\overline{\phi
}_{2h}^{n}$. Hence, we have just proved that for all $n$ there is a
unique solution $(\phi _{1h}^{n},\phi _{2h}^{n})$.
\end{proof}

\begin{lemma}
\label{lem10} Let $(\phi _{1h}^{n},\phi _{2h}^{n})\in W_{h}(\overline{D}%
_{1})\times V_{h}^{(1)}(\overline{D}_{2})$ be the solution to (\ref{fdm3})-(%
\ref{fdm4}). There exists a unique solution $\left( u_{h}^{n},v_{h\Delta
r}^{n}\right) \in V_{h}^{(1)}(\overline{D}_{1})\times V_{h\Delta r}(%
\overline{D}_{3})$ to the system (\ref{fdm1})-(\ref{fdm2}).
\end{lemma}

\begin{proof}
We start proving the existence of $u_{h}^{n}\in V_{h}^{(1)}(\overline{D}%
_{1}) $ as solution of (\ref{fdm1}). To this end, we write (\ref{fdm1}) as $%
F_{h}(u_{h}^{n})=0$, where $F_{h}:V_{h}^{(1)}(\overline{D}_{1})\rightarrow
V_{h}^{(1)}(\overline{D}_{1})$ is a continuous mapping defined by the
relation%
\begin{equation*}
\begin{array}{r}
\displaystyle\int_{D_{1}}F_{h}(\chi _{h})w_{h}dx=\displaystyle%
\int_{D_{1}}\left( \chi _{h}-u_{h}^{n-1}\right) w_{h}dx+\Delta
t\int_{D_{1}}k_{1}\frac{d\chi _{h}^{n}}{dx}\frac{dw_{h}}{dx}%
dx \\
\\
-\displaystyle\Delta t\int_{D_{1}}a_{1}J_{h}^{n}(\chi _{h})w_{h}dx=0\ \
\forall w_{h}\in V_{h}^{(1)}(\overline{D}_{1}),%
\end{array}%
\end{equation*}%
here,\ $J_{h}^{n}(\chi _{h})=J(x,\chi _{h},v_{sh},\phi _{1h}^{n},\phi _{2h}^{n},%
\overline{U}_{h}(v_{sh})),$ with $v_{sh}$ being picked up from
$S_{Q}^{\ast } $ because we assume that $v_{h\Delta r}^{n}$ belongs
to this space;
moreover, we also assume that $\chi _{h}$ is in $S_{P}$. According to Brower%
\'{}s fixed point theorem, the equation $F_{h}(\chi _{h})=0$ has a solution$%
\ \chi _{h}\in B_{q}:=\left\{ v_{h}\in V_{h}^{(1)}(\overline{D}%
_{1}):\left\Vert v_{h}\right\Vert _{L^{2}(D_{1})}\leq q\right\} $,
if $\int_{D_{1}}F_{h}(\chi _{h})\chi_{h}dx>0$ for $\left\Vert \chi
_{h}\right\Vert _{L^{2}(D_{1})}=q$. On account of the assumptions
$v_{sh}\in S_{Q}^{\ast }$ and $\chi _{h}\in S_{P}$, it follows that
there exists a constant $C_{1}=C_{1}(P,Q,K)$ such that\
$\int_{D_{1}}a_{1}J_{h}^{n}(\chi _{h})\chi_{h}dx\leq C_{1}\left\Vert
\chi _{h}\right\Vert _{L^{2}(D_{1})}$.
Hence,%
\begin{equation*}
\begin{array}{r}
\displaystyle\int_{D_{1}}F_{h}(\chi _{h})\chi _{h}dx\geq \left\Vert
\chi _{h}\right\Vert _{L^{2}(D_{1})}^{2}-\left\Vert
u_{h}^{n-1}\right\Vert
_{L^{2}(D_{1})}^{2}+\Delta tk_{0}\left\Vert \frac{d\chi _{h}^{n}}{%
dx}\right\Vert _{L^{2}(D_{1})}^{2}-\Delta tC_{1}\left\Vert \chi
_{h}\right\Vert _{L^{2}(D_{1})} \\
\\
\geq \left\Vert \chi _{h}\right\Vert _{L^{2}(D_{1})}^{2}-\left\Vert
u_{h}^{n-1}\right\Vert _{L^{2}(D_{1})}^{2}-\Delta tC_{1}\left( 1+\left\Vert
\chi _{h}\right\Vert _{L^{2}(D_{1})}\right) \left\Vert \chi _{h}\right\Vert
_{L^{2}(D_{1})}.%
\end{array}%
\end{equation*}%
Then, taking $\Delta t\leq \Delta t_{0}<1/C_{1}$, $\int_{D_{1}}F_{h}(\chi
_{h})\chi _{h}dx$ is positive for $\left\Vert \chi _{h}\right\Vert
_{L^{2}(D_{1})}$ sufficiently large. This shows the existence of the
solution $u_{h}^{n}\in V_{h}^{(1)}(\overline{D}_{1})$. Next, we prove the
uniqueness. To this end, we consider that there exist $X$ and $Y$ $\in
V_{h}^{(1)}(\overline{D}_{1})$ satisfying (\ref{fdm1}), so%
\begin{equation*}
\int_{D_{1}}\left( X-Y\right) w_{h}dx+\Delta t\int_{D_{1}}k_{1}\frac{%
d(X-Y)}{dx}\frac{dw_{h}}{dx}dx=\Delta t\int_{D_{1}}a_{1}\left(
J_{h}^{n}(X)-J_{h}^{n}(Y)\right) w_{h}dx\ \ \forall w_{h}\in
V_{h}^{(1)}(\overline{D}_{1}).
\end{equation*}%
Setting $w_{h}=X-Y$ and invoking the arguments of Lemmas \ref{lem3} and \ref%
{lem4} yields%
\begin{equation*}
\left\Vert X-Y\right\Vert _{L^{2}(D_{1})}^{2}+k_{1}\Delta t\left\Vert \frac{%
d(X-Y)}{dx}\right\Vert _{L^{2}(D_{1})}^{2}\leq \Delta
tC_{2}\left\Vert X-Y\right\Vert _{L^{2}(D_{1})}^{2},
\end{equation*}%
where the constant $C_{2}=C_{2}(P,Q,K)$. Thus, taking $\Delta t\leq \Delta
t_{0}<1/C_{2}$ it follows that $X=Y$. It remains to prove the existence and
uniqueness of $v_{h\Delta r}^{n}$, but the arguments to be used for such a
proof are the same as for $u_{h}^{n}$, so we omit them.
\end{proof}

\subsection{Error estimates for the fully discrete solution}

As in Section 4.3, we write for $t=t_{n}$%
\begin{equation}
\left\{
\begin{array}{l}
u^{n}-u_{h}^{n}=\rho _{u}^{n}+\theta _{u}^{n}, \\
\\
v^{n}-v_{h\Delta r}^{n}=v^{n}-I_{0}^{x}v^{n}+I_{0}^{x}\rho _{v}^{n}+\theta
_{v}^{n}.%
\end{array}%
\right.  \label{fdm7}
\end{equation}

\begin{theorem}
\label{Theorem 3} Let $(u_{h}^{n},v_{h\Delta r}^{n},\phi _{1}^{n},\phi
_{2}^{n})$ be the solution to (\ref{fdm1})-(\ref{fdm4}). Then, under proper
regularity assumptions there exists a constant $C$ such that for $\Delta t$
small%
\begin{equation}
\left\Vert u^{n}-u_{h}^{n}\right\Vert _{L^{2}(D_{1})}^{2}+\left\Vert
v^{n}-v_{h\Delta r}^{n}\right\Vert _{L^{2}(D_{2},L_{r}^{2}(0,R_{s}(\cdot
)))}^{2}+\Delta t\sum_{j=1}^{t_{n}}\left\Vert \Phi ^{j}-\Phi
_{h}^{j}\right\Vert _{V}^{2}d\tau \leq C(h^{2}+\Delta r^{2}+\Delta t^{2}).
\label{fdm8}
\end{equation}%
The constant $C$ is of the form $C(\Gamma )\exp
(C(\underline{k_{1}},\underline{k_{2}})t_{n}$, $C(\Gamma )$ being
another constant that depends on the exact solution
$(u,v,\phi_{1},\phi_{2})$, see (\ref{cgamma}) below.
\end{theorem}

\begin{proof}
Since $\rho _{u}^{n}$, $v^{n}-I_{0}^{x}v^{n}+I_{0}^{x}\rho _{v}^{n}$ and $%
\left\Vert \Phi ^{j}-\Phi _{h}^{j}\right\Vert _{V}^{2}$ are
estimated as in Section 4.3, we shall address our attention to the
estimations for $\theta _{u}^{n}$ and $\theta _{v}^{n}$. We start
with the calculation for $\theta
_{v}^{n}$. For this purpose, we recast (\ref{uv4}) for $t=t_{n}$ as%
\begin{equation}
\begin{array}{c}
\displaystyle\int_{D_{2}}\int_{0}^{R_{s}(x)}\left(\widetilde{\partial}_{t}I_{0}^{x}{v}%
^{n}w+k_{2}\frac{\partial I_{0}^{x}v^{n}}{\partial r}\frac{\partial w}{%
\partial r}\right) r^{2}drdx \\
\\
=-\displaystyle%
\int_{D_{2}}I_{0}^{x}(R_{s}^{2}(x)a_{2}^{-1}(x)F^{-1}J^{n}w(x;R_{s}(x)))dx \\
\\
+\displaystyle\int_{D_{2}}\int_{0}^{R_{s}(x)}\left(\widetilde{\partial}_{t}I_{0}^{x}{v}%
^{n}w-\frac{\partial I_{0}^{x}v^{n}}{\partial t}w\right) r^{2}drdx.%
\end{array}
\label{eef2}
\end{equation}%
Setting, as we did in Section 4.3, $v_{h\Delta
r}^{n}=I_{0}^{x}v^{n}-\left(
I_{0}^{x}\rho _{v}^{n}+\theta _{v}^{n}\right) $ in (\ref{fdm2}) and using (%
\ref{pre8}) and (\ref{eef2} ) yields for $t=t_{n}$
\begin{equation*}
\begin{array}{r}
\displaystyle\int_{D_{2}}\int_{0}^{R_{s}(x)}\left(\widetilde{\partial}_{t}\theta
_{v}^{n}w_{h\Delta r}+k_{2}\frac{\partial \theta _{v}^{n}}{\partial
r}\frac{
\partial w_{h\Delta r}}{\partial r}\right) r^{2}drdx=\lambda
\int_{D_{2}}\int_{0}^{R_{s}(x)}I_{0}^{x}\rho _{v}^{n}w_{h\Delta r}r^{2}drdx
\\
\\
-\displaystyle\int_{D_{2}}\int_{0}^{R_{s}(x)}\widetilde{\partial}_{t}I_{0}^{x}\rho
_{v}^{n}w_{h\Delta r}r^{2}drdx \\
\\
-\displaystyle\int_{D_{2}}I_{h}^{0}(R_{s}^{2}(x)a_{2}^{-1}(x)F^{-1}(J^{n}(x)-J_{h}^{n}(x))w_{hs}(x))dx
\\
\\
+\displaystyle\int_{D_{2}}\left(
R_{s}^{2}(x)a_{2}^{-1}(x)F^{-1}J_{h}^{n}(x)-I_{h}^{0}(R_{s}^{2}(x)a_{2}^{-1}(x)F^{-1}J_{h}^{n}(x))\right)
w_{hs}(x)dx,
\\
\\
-\displaystyle\int_{D_{2}}\int_{0}^{R_{s}(x)}\left(\widetilde{\partial}_{t}I_{0}^{x}v^{n}-\frac{\partial
I_{0}^{x}v^{n}}{\partial t}\right)
r^{2}drdx%
\end{array}%
\end{equation*}%
Letting $w_{h\Delta r}=\theta _{v}^{n}$, $w_{hs}=\theta _{vs}^{n}$,
and noting that for $a$ and $b$ real numbers,
$2(a-b)b=a^{2}-b^{2}\textcolor{blue}{-}(a-b)^{2}$,
it follows that%
\begin{equation}
\begin{array}{l}
\displaystyle\frac{1}{2}\widetilde{\partial}_{t}\left\Vert \theta
_{v}^{n}\right\Vert _{L^{2}(D_{2},L_{r}^{2}(0,R_{s}(\cdot
)))}^{2}+\displaystyle \underline{k_{2}}\left\Vert \frac{\partial
\theta _{v}^{n}}{\partial r}\right\Vert
_{L^{2}(D_{2},L_{r}^{2}(0,R_{s}(\cdot )))}^{2} \\
\\
\leq \lambda \left\Vert I_{0}^{x}\rho _{v}^{n}\right\Vert
_{L^{2}(D_{2},L_{r}^{2}(0,R_{s}(\cdot )))}\left\Vert \theta
_{v}^{n}\right\Vert _{L^{2}(D_{2},L_{r}^{2}(0,R_{s}(\cdot )))} \\
\\
+\displaystyle\left\Vert \widetilde{\partial}_{t}I_{0}^{x}\rho
_{v}^{n}\right\Vert _{L^{2}(D_{2},L_{r}^{2}(0,R_{s}(\cdot
)))}\left\Vert \theta
_{v}^{n}\right\Vert _{L^{2}(D_{2},L_{r}^{2}(0,R_{s}(\cdot )))} \\
\\
+C\left\Vert I_{h}^{0}\left( (J^{n}-J_{h}^{n})\theta _{vs}^{n}\right)
\right\Vert _{L^{1}(D_{2})}+C\left\Vert (J_{h}^{n}-I_{h}^{0}J_{h}^{n})\theta
_{vs}^{n}\right\Vert _{L^{1}(D_{2})} \\
\\
+\left\Vert \widetilde{\partial}_{t}I_{0}^{x}v^{n}-\frac{\partial I_{0}^{x}v^{n}}{%
\partial {t}}\right\Vert _{L^{2}(D_{2},L_{r}^{2}(0,R_{s}(\cdot
)))}\left\Vert \theta _{v}^{n}\right\Vert
_{L^{2}(D_{2},L_{r}^{2}(0,R_{s}(\cdot )))}\equiv \sum_{i=1}^{5}R_{i}^{n}.%
\end{array}
\label{eef4}
\end{equation}%
We bound the right hand side of this inequality applying the same arguments
as in (\ref{uv7}). Thus, we have that%
\begin{equation*}
R_{1}^{n}\leq C\Delta r^{4}\left\Vert v^{n}\right\Vert
_{L^{2}(D_{2},H_{r}^{2}(0,R_{s}(\cdot )))}^{2}+C\left\Vert \theta
_{v}^{n}\right\Vert _{L^{2}(D_{2},L_{r}^{2}(0,R_{s}(\cdot )))}^{2}
\end{equation*}

\begin{equation*}
R_{2}^{n}\leq \frac{C\Delta r^{4}}{\Delta t}\int_{t_{n-1}}^{t_{n}}\left\Vert
\frac{\partial v}{\partial t}\right\Vert
_{L^{2}(D_{2},H_{r}^{2}(0,R_{s}(\cdot )))}^{2}dt+C\left\Vert \theta
_{v}^{n}\right\Vert _{L^{2}(D_{2},L_{r}^{2}(0,R_{s}(\cdot )))}^{2}.
\end{equation*}

\begin{equation*}
R_{3}^{n}\leq C\left\Vert J^{n}-J_{h}^{n}\right\Vert
_{L^{2}(D_{2})}^{2}+C(\epsilon )\left\Vert \theta _{v}^{n}\right\Vert
_{L^{2}(D_{2},L_{r}^{2}(0,R_{s}(\cdot )))}^{2}+\epsilon \left\Vert \frac{%
\partial \theta _{v}^{n}}{\partial r}\right\Vert
_{L^{2}(D_{2},L_{r}^{2}(0,R_{s}(\cdot )))}^{2}.
\end{equation*}

\begin{equation*}
\begin{array}{r}
R_{4}^{n}\leq C\left\Vert J^{n}-J_{h}^{n}\right\Vert
_{L^{2}(D_{2})}^{2}+C(\epsilon )\left\Vert \theta _{v}^{n}\right\Vert
_{L^{2}(D_{2},L_{r}^{2}(0,R_{s}(\cdot )))}^{2}+\epsilon \displaystyle%
\left\vert \frac{\partial \theta _{v}^{n}}{\partial r}\right\vert
_{L^{2}(D_{2},L_{r}^{2}(0,R_{s}(\cdot )))}^{2} \\
\\
+Ch^{2}\displaystyle\left\Vert \frac{\partial J^{n}}{\partial x}\right\Vert
_{L^{2}(D_{2})}^{2}.%
\end{array}%
\end{equation*}%
In both $R_{3}^{n}$ and $R_{4}^{n}$ the term $C\left\Vert
J^{n}-J_{h}^{n}\right\Vert _{L^{2}(D_{2})}^{2}$ is bounded by Lemma \ref%
{lem6} for $t=t_{n}$; thus, using the notation $\Gamma =(u,v,v_{s},\phi
_{1},\phi _{2})$, we can set that%
\begin{equation}
\begin{array}{l}
C\left\Vert J^{n}-J_{h}^{n}\right\Vert _{L^{2}(D_{2})}^{2}\leq
C(\Gamma )(h^{2}+\Delta r^{2})+C\left\Vert \theta
_{u}^{n}\right\Vert
_{L^{2}(D_{1})}^{2} \\
\\
+C(\epsilon )\left\Vert \theta _{v}^{n}\right\Vert
_{L^{2}(D_{2},L_{r}^{2}(0,R_{s}(\cdot )))}^{2}+\epsilon \displaystyle%
\left\Vert \frac{\partial \theta _{v}^{n}}{\partial r}\right\Vert
_{L^{2}(D_{2},L_{r}^{2}(0,R_{s}(\cdot )))}^{2},%
\end{array}
\label{eef5}
\end{equation}%
where the constant $C(\Gamma )$ is given as%
\begin{equation} \label{cgamma}
\begin{array}{c}
C(\Gamma )=C\max \left( \left\Vert \phi _{1}\right\Vert _{L^{\infty }(0,T_{%
\mathrm{end}};H^{2}(D_{1}))},\left\Vert \phi _{2}\right\Vert
_{L^{\infty }(0,T_{\mathrm{end}};H^{2}(D_{2}))},\left\Vert
v_{s}\right\Vert _{L^{\infty
}(0,T_{\mathrm{end}};H^{1}(D_{2}))},\right. \\
\\
\left. \left\Vert u\right\Vert _{L^{\infty }(0,T_{\mathrm{end}%
};H^{2}(D_{1}))},\left\Vert v\right\Vert _{L^{\infty }(0,T_{\mathrm{end}%
};L^{2}(D_{2};H_{r}^{2}(0,R_{s}(\cdot))))}\right) .%
\end{array}%
\end{equation}%
Hence, we can write%
\begin{equation*}
\begin{array}{r}
R_{3}^{n}+R_{4}^{n}\leq C(\Gamma )(h^{2}+\Delta r^{2})+Ch^{2}\displaystyle%
\left\Vert \frac{\partial J^{n}}{\partial x}\right\Vert
_{L^{2}(D_{2})}^{2}+C\left\Vert \theta _{u}^{n}\right\Vert
_{L^{2}(D_{1})}^{2} \\
\\
+C(\epsilon )\left\Vert \theta _{v}(t)\right\Vert
_{L^{2}(D_{2},L_{r}^{2}(0,R_{s}(\cdot )))}^{2}+\epsilon \displaystyle%
\left\Vert \frac{\partial \theta _{v}(t)}{\partial r}\right\Vert
_{L^{2}(D_{2},L_{r}^{2}(0,R_{s}(\cdot )))}^{2}.%
\end{array}%
\end{equation*}%
To estimate the term $R_{5}^{n}$, we notice that by approximation
theory
\begin{equation*}
\left\Vert \widetilde{\partial}_{t}I_{0}^{x}v^{n}-\frac{\partial I_{0}^{x}v^{n}}{%
\partial {t}}\right\Vert _{L^{2}(D_{2},L_{r}^{2}(0,R_{s}(\cdot )))}\leq
C\left\Vert \widetilde{\partial}_{t}v^{n}-\frac{\partial v^{n}}{\partial {t}}%
\right\Vert _{L^{2}(D_{2},L_{r}^{2}(0,R_{s}(\cdot )))}
\end{equation*}%
and%
\begin{equation*}
\widetilde{\partial}_{t}v^{n}-\frac{\partial v^{n}}{\partial {t}}=\frac{-1}{\Delta t}%
\int_{t_{n-1}}^{t_{n}}(t-t_{n-1})\frac{\partial ^{2}v}{\partial
t^{2}}dt,
\end{equation*}%
so,%
\begin{equation*}
\left\Vert \widetilde{\partial}_{t}I_{0}^{x}v^{n}-\frac{\partial I_{0}^{x}v^{n}}{%
\partial {t}}\right\Vert _{L^{2}(D_{2},L_{r}^{2}(0,R_{s}(\cdot )))}\leq
C\left( \Delta t\int_{t_{n-1}}^{t_{n}}\left\Vert \frac{\partial ^{2}v}{%
\partial {t}^{2}}\right\Vert _{L^{2}(D_{2},L_{r}^{2}(0,R_{s}(\cdot
)))}^{2}dt\right) ^{1/2}.
\end{equation*}%
Applying Young inequality yields%
\begin{equation*}
R_{5}^{n}\leq C\Delta t\int_{t_{n-1}}^{t_{n}}\left\Vert \frac{\partial ^{2}v%
}{\partial {t}^{2}}\right\Vert _{L^{2}(D_{2},L_{r}^{2}(0,R_{s}(\cdot
)))}^{2}dt+C\left\Vert \theta _{v}^{n}\right\Vert
_{L^{2}(D_{2},L_{r}^{2}(0,R_{s}(\cdot )))}^{2}
\end{equation*}%
Collecting these bounds in (\ref{eef4}) and letting $\epsilon
=\underline{k_{2}}/2$
yields%
\begin{equation}
\begin{array}{r}
\displaystyle\left\Vert \theta _{v}^{n}\right\Vert
_{L^{2}(D_{2},L_{r}^{2}(0,R_{s}(\cdot )))}^{2}+\displaystyle\Delta
t\underline{k_{2}}\left\Vert \frac{\partial \theta
_{v}^{n}}{\partial r}\right\Vert
_{L^{2}(D_{2},L_{r}^{2}(0,R_{s}(\cdot )))}^{2}\leq
\displaystyle\left\Vert \theta _{v}^{n-1}\right\Vert
_{L^{2}(D_{2},L_{r}^{2}(0,R_{s}(\cdot
)))}^{2}+F_{v}^{n} \\
\\
+C(\underline{k_{2}})\Delta t\left( \left\Vert \theta
_{v}^{n}\right\Vert _{L^{2}(D_{2},L_{r}^{2}(0,R_{s}(\cdot
)))}^{2}+\left\Vert \theta
_{u}^{n}\right\Vert _{L^{2}(D_{1})}^{2}\right) ,%
\end{array}
\label{eef6.0}
\end{equation}%
where%
\begin{equation}
\begin{array}{c}
F_{v}^{n}=\Delta t\left( C(\Gamma )(h^{2}+\Delta r^{2})+Ch^{2}\displaystyle%
\left\Vert \frac{\partial J^{n}}{\partial x}\right\Vert
_{L^{2}(D_{2})}^{2}+C\Delta r^{4}\left\Vert v^{n}\right\Vert
_{L^{2}(D_{2},H_{r}^{2}(0,R_{s}(\cdot )))}^{2}\right) \\
\\
+C\Delta r^{4}\displaystyle\int_{t_{n-1}}^{t_{n}}\left\Vert \frac{\partial v%
}{\partial t}\right\Vert _{L^{2}(D_{2},H_{r}^{2}(0,R_{s}(\cdot
)))}^{2}dt+C\Delta t^{2}\displaystyle\int_{t_{n-1}}^{t_{n}}\left\Vert \frac{%
\partial ^{2}v}{\partial {t}^{2}}\right\Vert
_{L^{2}(D_{2},L_{r}^{2}(0,R_{s}(\cdot )))}^{2}dt.%
\end{array}
\label{eef6.1}
\end{equation}%

To calculate an estimate for $\theta _{u}^{n}$, we observe that
subtracting (\ref{fdm1}) from (\ref{w1}) and setting $\theta
_{u}^{n}=e_{u}^{n}-\rho
_{u}^{n}$ it follows that%
\begin{equation*}
\begin{array}{r}
\displaystyle\int_{D_{1}}\widetilde{\partial}_{t}\theta _{u}^{n}w_{h}dx+\displaystyle\int_{D_{1}}k_{1}%
\frac{d\theta _{u}^{n}}{dx}\frac{dw_{h}}{dx}dx=\lambda
\int_{D_{1}}\rho _{u}^{n}w_{h}dx-\int_{D_{1}}\partial \rho
_{u}^{n}w_{h}dx
\\
\\
+\displaystyle\int_{D_{1}}a_{1}\left( J^{n}-J_{h}^{n}\right) w_{h}dx+%
\displaystyle\int_{D_{1}}\left( \widetilde{\partial}_{t}u^{n}-\frac{\partial u^{n}}{%
\partial t}\right) w_{h}dx.%
\end{array}%
\end{equation*}%
Letting $w_{h}=\theta _{u}^{n}$ yields%
\begin{equation*}
\begin{array}{r}
\displaystyle\frac{1}{2}\widetilde{\partial}_{t}\left\Vert \theta
_{u}^{n}\right\Vert
_{L^{2}(D_{1})}^{2}+\underline{k_{1}}\displaystyle\left\Vert
\frac{d\theta
_{u}^{n}}{dx}\right\Vert _{L^{2}(D_{1})}^{2}\leq C\displaystyle%
\left( \left\Vert \rho _{u}^{n}\right\Vert _{L^{2}(D_{1})}^{2}+\frac{1}{%
\Delta t}\int_{t_{n-1}}^{t_{n}}\left\Vert \frac{\partial \rho _{u}}{\partial
t}\right\Vert _{L^{2}(D_{1})}^{2}dt\right) \\
\\
+C\Delta t\displaystyle\int_{t_{n-1}}^{t_{n}}\left\Vert \frac{\partial ^{2}u%
}{\partial ^{2}t}\right\Vert _{L^{2}(D_{1})}^{2}dt+C\left\Vert
J^{n}-J_{h}^{n}\right\Vert _{L^{2}(D_{1})}^{2}+C\left\Vert \theta
_{u}^{n}\right\Vert _{L^{2}(D_{1})}^{2}.%
\end{array}%
\end{equation*}%
Then by virtue of (\ref{pre5}) and Lemma \ref{lem6} it follows that%
\begin{equation} \label{eef6.01}
\begin{array}{l}
\left\Vert \theta _{u}^{n}(t)\right\Vert _{L^{2}(D_{1})}^{2}+\Delta
t\underline{k_{1}}
\displaystyle\left\Vert \frac{d\theta _{u}^{n}}{dx}%
\right\Vert _{L^{2}(D_{1})}^{2}\leq \left\Vert \theta
_{u}^{n-1}(t)\right\Vert _{L^{2}(D_{1})}^{2}+F_{u}^{n} \\
\\
+C(\underline{k_{1}},\underline{k_{2}})\Delta t\left( \left\Vert
\theta _{v}^{n}\right\Vert _{L^{2}(D_{2},L_{r}^{2}(0,R_{s}(\cdot
)))}^{2}+\left\Vert \theta _{u}^{n}\right\Vert
_{L^{2}(D_{1})}^{2}\right)+\displaystyle\frac{\underline{k_{2}}}{2}\left\Vert
\frac{\partial {\theta_{v}^{n}}}{\partial r}\right\Vert
_{L^{2}(D_{2},L_{r}^{2}(0,R_{s}(\cdot )))}^{2}
\end{array}%
\end{equation}%
where%
\begin{equation}
F_{u}^{n}=\Delta tC(\Gamma )(h^{2}+\Delta r^{2})+Ch^{4}\displaystyle%
\int_{t_{n-1}}^{t_{n}}\left\Vert \frac{\partial u}{\partial t}\right\Vert
_{H^{2}(D_{1})}^{2}dt+C\Delta t^{2}\displaystyle\int_{t_{n-1}}^{t_{n}}\left%
\Vert \frac{\partial ^{2}u}{\partial ^{2}t}\right\Vert _{L^{2}(D_{1})}^{2}dt
\label{eef7}
\end{equation}

It remains to estimate $\left\Vert \Phi ^{n}-\Phi _{h}^{n}\right\Vert _{V}$.
Returning to the proof of Theorem \ref{Theorem 2} we have that for $t=t_{n}$%
\begin{equation}
\left\Vert \Phi ^{n}-\Phi _{h}^{n}\right\Vert _{V}^{2}\leq C(\Gamma
)h^{2}+C(\epsilon )\left( \left\Vert \theta _{u}^{n}\right\Vert
_{L^{2}(D_{1})}^{2}+\left\Vert \theta _{v}^{n}\right\Vert
_{L^{2}(D_{2},L_{r}^{2}(0,R_{s}(\cdot )))}^{2}\right) +\epsilon \displaystyle%
\left\Vert \frac{\partial \theta _{v}^{n}}{\partial r}\right\Vert
_{L^{2}(D_{2},L_{r}^{2}(0,R_{s}(\cdot )))}^{2}.  \label{eef8}
\end{equation}%
Thus, setting $\epsilon=\underline{k_{2}}/2$ in (\ref{eef8}) and
adding (\ref{eef6.0}), (\ref{eef6.01}) and (\ref{eef8}) we obtain
that
\begin{equation*}
\begin{array}{l}
\left( 1-C(\underline{k_{1}},\underline{k_{2}})\Delta t\right)
\left( \left\Vert \theta _{u}^{n}\right\Vert
_{L^{2}(D_{1})}^{2}+\displaystyle\left\Vert \theta
_{v}^{n}\right\Vert _{L^{2}(D_{2},L_{r}^{2}(0,R_{s}(\cdot
)))}^{2}\right) +\Delta t\left\Vert \Phi ^{n}-\Phi
_{h}^{n}\right\Vert _{V}^{2}\leq
F_{v}^{n}+F_{u}^{n} \\
\\
+\left\Vert \theta _{u}^{n-1}\right\Vert _{L^{2}(D_{1})}^{2}+\displaystyle%
\left\Vert \theta _{v}^{n-1}\right\Vert
_{L^{2}(D_{2},L_{r}^{2}(0,R_{s}(\cdot )))}^{2}%
\end{array}%
\end{equation*}%
For $\Delta t$ small%
\begin{equation*}
\begin{array}{r}
\left\Vert \theta _{u}^{n}\right\Vert _{L^{2}(D_{1})}^{2}+\displaystyle%
\left\Vert \theta _{v}^{n}\right\Vert _{L^{2}(D_{2},L_{r}^{2}(0,R_{s}(\cdot
)))}^{2}+\Delta t\left\Vert \Phi ^{n}-\Phi _{h}^{n}\right\Vert _{V}^{2}\leq
C\left( F_{v}^{n}+F_{u}^{n}\right) \\
\\
\left( 1+C(\underline{k_{1}},\underline{k_{2}})\Delta t\right)
\left( \left\Vert \theta _{u}^{n-1}\right\Vert
_{L^{2}(D_{1})}^{2}+\displaystyle\left\Vert \theta
_{v}^{n-1}\right\Vert _{L^{2}(D_{2},L_{r}^{2}(0,R_{s}(\cdot )))}^{2}\right) .%
\end{array}%
\end{equation*}%
Hence, by repeated application and taking $\left\Vert \theta
_{u}^{0}\right\Vert _{L^{2}(D_{1})}^{2}+\left\Vert \theta
_{v}^{0}\right\Vert _{L^{2}(D_{2},L_{r}^{2}(0,R_{s}(\cdot
)))}^{2}=0$, it results that
\begin{equation*}
\left\Vert \theta _{u}^{n}\right\Vert _{L^{2}(D_{1})}^{2}+\left\Vert \theta
_{v}^{n}\right\Vert _{L^{2}(D_{2},L_{r}^{2}(0,R_{s}(\cdot )))}^{2}+\Delta
t\sum_{j=1}^{n}\left\Vert \Phi ^{j}-\Phi _{h}^{j}\right\Vert _{V}^{2}\leq C%
\displaystyle\sum_{j=1}^{n}\frac{F_{v}^{j}+F_{u}^{j}}{\left(
1+C(\underline{k_{1}},\underline{k_{2}}))\Delta t\right) ^{j-n}}.
\end{equation*}%
Noting that $\left( 1+C(k_{1},k_{2})\Delta t\right) \leq
e^{C(k_{1},k_{2})\Delta t}$, then we can write%
\begin{equation*}
\begin{array}{c}
\displaystyle\sum_{j=1}^{n}\frac{F_{v}^{j}+F_{u}^{j}}{\left(
1+C(\underline{k_{1}},\underline{k_{2}})\Delta t\right) ^{j-n}}\leq
e^{C(\underline{k_{1}},\underline{k_{2}})t_{n}}\sum_{j=1}^{n}F_{v}^{j}+F_{u}^{j}\ \ \left( \mathrm{%
by\ (\ref{eef6.1})\ and\ (\ref{eef7})}\right) \\
\\
\ \ \ \ \ \ \ \ \ \ \ \ \ \ \ \ \ \ \ \ \ \leq C(\Gamma
)e^{C(\underline{k_{1}},\underline{k_{2}})t_{n}}\left( h^{2}+\Delta r^{2}+\Delta t^{2}\right) .%
\end{array}%
\end{equation*}%
This completes the proof
\end{proof}

\section{Acknowledgements}

This research has been partially funded by grant PGC-2018-097565-B100 of
Ministerio de Ciencia, Innovaci\'{o}n y Universidades of Spain.

\end{document}